\documentclass{article}
\usepackage{amsmath}
\usepackage{amssymb}
\usepackage{slashed}
\usepackage[letterpaper]{hyperref}

\newtheorem{theorem}{Theorem}[section]

\newtheorem{corollary}[theorem]{Corollary}

\newtheorem{definition}[theorem]{Definition}

\newtheorem{lemma}[theorem]{Lemma}

\newtheorem{proposition}[theorem]{Proposition}
\newtheorem{remark}[theorem]{Remark}

\newenvironment{proof}[1][Proof]{\noindent\textbf{#1.} }{\ \rule{0.5em}{0.5em}}
\renewcommand{\theequation}{\thesection.\arabic{equation}}
\input{tcilatex}
\begin{document}

\title{Estimates and monotonicity for a heat flow of isometric $G_{2}$%
-structures}
\author{Sergey Grigorian \\
School of Mathematical and Statistical Sciences\\
University of Texas Rio Grande Valley\\
1201 W. University Drive\\
Edinburg, TX 78539}
\maketitle

\begin{abstract}
Given a $7$-dimensional compact Riemannian manifold $\left( M,g\right) $
that admits $G_{2}$-structure, all the $G_{2}$-structures that are
compatible with the metric $g$ are parametrized by unit sections of an
octonion bundle over $M$. We define a natural energy functional on unit
octonion sections and consider its associated heat flow. The critical points
of this functional and flow precisely correspond to $G_{2}$-structures with
divergence-free torsion. In this paper, we first derive estimates for
derivatives of \ $V\left( t\right) $ along the flow and prove that the flow
exists as long as the torsion remains bounded. We also prove a monotonicity
formula and an $\varepsilon $-regularity result for this flow. Finally,
we show that within a metric class of $G_{2}$-structures that contains a
torsion-free $G_{2}$-structure, under certain conditions, the flow will
converge to a torsion-free $G_{2}$-structure.
\end{abstract}

\section{Introduction}

A fundamental problem in the study of $7$-dimensional manifolds with $G_{2}$%
-structures is the question of general existence conditions for torsion-free 
$G_{2}$-structures, which are the ones that correspond to metrics with
holonomy contained in $G_{2}$. One of the possible approaches is to try and
construct a flow of $G_{2}$-structures which under certain conditions would
converge to torsion-free $G_{2}$-structures. This approach was originally
pioneered by Robert Bryant \cite{bryant-2003} when he introduced the
Laplacian flow of closed $G_{2}$-structures, i.e. ones for which the
defining $3$-form $\varphi $ is closed. Later, Karigiannis, McKay, and Tsui 
\cite{KarigiannisMcKayTsui} introduced a similar flow, known as the \emph{%
Laplacian coflow}, for co-closed $G_{2}$-structures. It has some similar
properties to Bryant's flow - its stationary points are precisely
torsion-free $G_{2}$-structures, and it may be interpreted as the gradient
flow of the volume functional. However, as discovered in \cite%
{GrigorianCoflow}, it has a crucial deficiency in that it is non-parabolic.
In that paper, the author attempted to rectify the coflow by introducing the 
\emph{modified coflow}, which has an additional term that changes the sign
of the term that made the coflow non-parabolic, but would still preserve the
co-closed condition. However, the new flow lacks some of nicer features of
the original coflow - in particular, it has additional non-torsion-free
stationary points and it is not known if it can be written as a gradient
flow of some functional. One of the advantages of working with co-closed $%
G_{2}$-structures is that they are generally more abundant than closed ones.
An application of the $h$-principle in \cite{CrowleyNordstrom1} shows that
any compact manifold that admits $G_{2}$-structures will also admit
co-closed $G_{2}$-structures. Therefore, it is very important to understand
under which conditions it is possible to deform a co-closed $G_{2}$%
-structure to a torsion-free one.

More specifically, given a co-closed $G_{2}$-structure, i.e. one where $\psi
:=\ast \varphi $ is a closed $4$-form, the linearization at $\psi $ of the
corresponding Hodge Laplacian is an indefinite operator. This is in contrast
to the closed case, i.e. when $d\varphi =0$, where the linearization at $%
\varphi $ of $\Delta _{\varphi }\ $is a semi-definite operator, which can
then made strongly elliptic by the addition of a Lie derivative term to take
into account diffeomorphism invariance. In the co-closed case, the term that
causes $\Delta _{\psi }$ to be indefinite is $\pi _{7}\left( \Delta _{\psi
}\psi \right) ,$ which is the component of $\Delta _{\psi }\psi $ in the $7$%
-dimensional representation $\Lambda _{7}^{4}$ of $G_{2}.$ This term is
however determined by $\func{div}T$ - the divergence of the torsion \cite%
{GrigorianCoflow,GrigorianFlowSurvey}. It should be noted that for closed $%
G_{2}$-structures, $\func{div}T$ always vanishes \cite{bryant-2003}, so that
is why this issue doesn't arise in that case. Therefore, the condition $%
\func{div}T=0$ may be thought of as another \textquotedblleft
gauge-fixing\textquotedblright\ condition to make $\Delta _{\psi }\psi $
elliptic. From the point of view of $G_{2}$-structure coflows, the condition 
$\func{div}T=0$ makes the original and modified coflows equal to leading
order. Therefore, these considerations make it very important to understand
this divergence-free torsion property and in particular, under which
conditions do $G_{2}$-structures with $\func{div}T=0$ exist.

Another motivation for looking at divergence-free torsion comes from the
following observation. As noted above, $\func{div}T$ enters the $\Lambda
_{7}^{4}$ part of $\Delta _{\psi }\psi $. However, it is known \cite%
{karigiannis-2005-57} that deformations of $\varphi $ along $\Lambda
_{7}^{3},$ and equivalently of $\psi $ along $\Lambda _{7}^{4},$ keep the
metric unchanged and simply deform the $G_{2}$-structure within a fixed
metric class. Therefore, fixing $\func{div}T=0$ essentially corresponds to
taking particular a representatives of the metric class. Indeed, in an
investigation of isometric $G_{2}$-structures (that is, ones that are
compatible with the same metric) in \cite{GrigorianOctobundle}, it was found
that on a compact manifold, $G_{2}$-structures with $\func{div}T=0$ are
precisely the critical points of the $L^{2}$-norm of the torsion when
restricted to a fixed metric class. In \cite{GrigorianOctobundle} this
functional was also reformulated as an energy functional $\mathcal{E=}%
\int_{M}\left\vert DV\right\vert ^{2}\func{vol}$ for unit octonion sections
that parametrize isometric $G_{2}$-structures, were $V$ is a unit octonion
section and $D$ is the octonion covariant derivative defined with respect to
some fixed background $G_{2}$-structure. This allowed to rewrite the
condition $\func{div}T=0$ as a semilinear elliptic equation for octonion
sections and similarly, the negative gradient flow of $\mathcal{E}$ then
becomes a semilinear heat equation:%
\begin{equation}
\frac{\partial V}{\partial t}=\Delta _{D}V+\left\vert DV\right\vert ^{2}V
\label{flow}
\end{equation}%
where $\Delta _{D}=-D^{\ast }D$ is the Laplacian operator corresponding to $%
D $.

Given that (\ref{flow}) precisely corresponds to the $\Lambda _{7}^{4}$
component of the Laplacian coflow, it is crucial to understand its
properties in more detail. In particular, it is expected that at least under
some conditions, it should converge to a $G_{2}$-structure with $\func{div}%
T=0.$ In future work, this may be used as a gauge-fixing condition that
could relate the original coflow and the modified one.

It is noteworthy that this flow has remarkable similarities to the harmonic
heat flow and the Yang-Mills flow. Just as these two classical flows, (\ref%
{flow}) appears as the gradient flow of an energy functional and in the
analysis it becomes clear that many of the tools used for the harmonic heat
flow and the Yang-Mills flow can be adapted in this setting as well. As
such, it is another example of a flow that doesn't change the geometry of
the underlying space (as opposed to the Ricci flow and aforementioned
Laplacian flows of $G_{2}$-structures), but is still fundamentally related
to the geometry.

In this paper, we first give a brief overview of $G_{2}$-structures and
octonion bundles in Sections \ref{secg2struct} and \ref{secOctoCovDiv}.
Then, in Section \ref{secEnergy} we reintroduce the energy functional for
octonions and consider some of its properties. In Section \ref{secHeat}, we
then work out estimates for the flow (\ref{flow}). For convenience, we
introduce the quantity $\Lambda \left( x,t\right) =\left\vert DV\left(
x,t\right) \right\vert ^{2}=\left\vert T^{\left( V\right) }\right\vert ^{2},$
where $T^{\left( V\right) }$ is the torsion of the $G_{2}$-structure that
corresponds to the octonion section $V$. We work out the evolution of
derivatives of $V$ and prove the following:

\begin{enumerate}
\item If $V\left( t\right) $ is a smooth solution to (\ref{flow}) on a
finite maximal time interval $[0,t_{\max })$, then for any $0\leq t<t_{\max
} $, 
\begin{equation}
\Lambda \left( t\right) \leq \frac{2R_{1}}{\left( 1+\frac{2R_{1}}{\Lambda
_{0}+R_{2}}\right) e^{-4R_{1}t}-1}-R_{2}
\end{equation}%
where $\Lambda \left( t\right) =\sup_{x\in M}\Lambda \left( x,t\right) $ and 
$R_{1},$ $R_{2}$ are some constants that depend on the curvature and the
background $G_{2}$-structure.

\item If $V\left( t\right) $ is a smooth solution to (\ref{flow}) on a
finite maximal time interval $[0,t_{\max }),$ and $\Lambda \left( t\right) $
is bounded, then all derivatives of $V$ also remain bounded.

\item As long as $\Lambda \left( t\right) ,$ and hence $\left\vert T^{\left(
V\right) }\right\vert $, remains bounded, there will exist a smooth solution 
$V\left( t\right) $ to the flow (\ref{flow}).
\end{enumerate}

The methods used here are similar to what Lotay and Wei \cite{LotayWei1}
used for the Laplacian flow of $G_{2}$-structures, Weinkove \cite{WeinkoveYM}
used for the Yang-Mills flow, Grayson and Hamilton \cite{GraysonHamilton}
used for the harmonic map flow, and which originally Shi \cite{ShiEstimate}
introduced for the Ricci flow.

In Section \ref{secMonotonicty}, given a solution $V$ of (\ref{flow}), we
define the quantity 
\begin{equation}
Z\left( t\right) =\left( t_{0}-t\right) \int_{M}\left\vert DV\right\vert
^{2}k\func{vol}
\end{equation}%
where $k$ is a positive scalar solution of the backwards heat equation,
evolving backwards in time from $t=t_{0}$, and in Theorem \ref%
{thmMonotonicity}, we prove that $Z\left( t\right) $ satisfies an almost
monotonicity formula. While $Z\left( t\right) $ is not strictly monotonic
along the flow, it is well-behaved and can be controlled. In particular, we
show that for $t\geq \tau $, 
\begin{equation}
Z\left( t\right) \leq CZ\left( \tau \right) +C\left( t-\tau \right) \left( 
\mathcal{E}_{0}+\mathcal{E}_{0}^{\frac{1}{2}}\right)
\end{equation}%
where $C$ is some constant that depends on the geometry of the manifold and $%
\mathcal{E}_{0}$ is the initial value of the functional $\mathcal{E}.$ This
is similar to the monotonicity results obtained by Hamilton for the harmonic
map heat flow and the Yang-Mills flow in \cite{HamiltonMonotonicity}. Other
versions of monotonicity results had been obtained for the harmonic map flow
in \cite{ChenDingHM,StruweChen,StruweHM1} and for the Yang-Mills flow in 
\cite{ChenShenYM,HongTianYM,KelleherStreetsYM1}.

In Section \ref{secEpsReg}, we define the $\mathcal{F}$-functional, which
essentially replaces $k$ in $Z$ by the heat kernel of the backwards heat
equation. Applying the monotonicity formula allows us to prove an $%
\varepsilon $-regularity result for solutions of (\ref{flow}) in Theorem \ref%
{thmEpsRegular}, which says that if $\mathcal{F}\leq \varepsilon $ for some $%
\varepsilon >0$, then the flow may always be smoothly extended. This then
leads on to global existence of solutions for sufficiently small initial
energy density $\Lambda _{0}=\left\vert DV\left( 0\right) \right\vert ^{2}$.
This again builds upon prior work on the harmonic map heat flow and the
Yang-Mills flow. An elliptic version of $\varepsilon $-regularity for
harmonic maps was originally introduced by Schoen and Uhlenbeck in \cite%
{SchoenUhlenbeckHM1}, and parabolic versions were given by Struwe \cite%
{StruweHM1}, Chen and Ding \cite{ChenDingHM}, Grayson and Hamilton \cite%
{GraysonHamilton}. For the Yang-Mills flow, $\varepsilon $-regularity
results were given by Chen and Shen in \cite{ChenShenYM} and Weinkove in 
\cite{WeinkoveYM}.

In Section \ref{secrealim}, we consider a special case when the flow takes
place in the presence of a torsion-free $G_{2}$-structure, that is, the
metric has holonomy contained in $G_{2}$. In that case, we can take the
background $G_{2}$-structure to be torsion-free, and consider the flow (\ref%
{flow}) starting from an arbitrary octonion section. That particular
torsion-free $G_{2}$-structure is then represented by constant octonion
sections $\pm 1$. Therefore, it makes sense to decompose the unit octonion $%
V\left( t\right) $ into real and imaginary parts and then the evolution of
the real part $f\left( t\right) =\func{Re}V\left( t\right) $ is of
particular interest. Indeed, we find that $f^{2}$ satisfies the Minimum
Principle, so that pointwise, it is bounded below by its infimum at $t=0\,,$
and moreover, if initially $\inf \left\vert f\left( 0\right) \right\vert >0$%
, then the $L^{1}$-norm of $f$ grows monotonically along the flow, with the
time derivative bounded below by a constant multiple of $\mathcal{E}$. This
is significant because clearly $\left\vert f\left( t\right) \right\vert $ is
bounded above by $1$, and hence if the flow exists for all $t\geq 0$ (such
as under the condition from Section \ref{secEpsReg}), then it must reach a
torsion-free $G_{2}$-structure, i.e. a global minimum of $\mathcal{E}$. This
is again very similar to the behavior of the harmonic map heat flow, where
if the flow satisfies a small initial energy condition and the initial map
is homotopic to a constant map, then the heat flow will converge to a
constant map \cite{ChenDingHM}.

Note that two days after the initial version of this paper appeared on
arXiv, a paper by Dwivedi, Gianniotis, and Karigiannis \cite{DGKisoflow}
that has a substantial but independent overlap with this paper has also been
posted. However, while a number of conclusions and techniques are similar,
the points of view on the flow (\ref{heatflow}) are different. In this
paper, we regard this as a flow of octonion sections, while in \cite%
{DGKisoflow} a more traditional geometric flows approach is used. Both
approaches are valuable and complementary and provide different perspectives
on the same phenomenon. Since the appearance of the initial versions of these two papers, there has been a very useful cross-pollination of ideas and in the final version of the present paper in some instances we allude to \cite{DGKisoflow} for additional clarity and completeness.

Even more recently, a preprint by Loubeau and S\'{a}%
\ Earp \cite{SaEarpLoubeau} appeared, where a similar flow is studied but from yet
another point of view. In \cite{SaEarpLoubeau}, a more general concept of a
harmonic geometric structure is defined, which in the $G_{2}$ case reduces
to critical points of the functional $\mathcal{E}$ - that is, $G_{2}$%
-structures with divergence-free torsion. Similarly, the harmonic flow of
geometric structures then reduces to the flow (\ref{flow}) in the $G_{2}$
case. 

\subsubsection*{Acknowledgements}

This work was supported by the National Science Foundation [DMS-1811754].
The author also thanks the referee for very helpful remarks.

\section{$G_{2}$-structures}

\setcounter{equation}{0}\label{secg2struct}The 14-dimensional group $G_{2}$
is the smallest of the five exceptional Lie groups and is closely related to
the octonions, which is the noncommutative, nonassociative, $8$-dimensional
normed division algebra. In particular, $G_{2}$ can be defined as the
automorphism group of the octonion algebra. Given the octonion algebra $%
\mathbb{O}$, there exists a unique orthogonal decomposition into a real
part, that is isomorphic to $\mathbb{R},$ and an \emph{imaginary }(or \emph{%
pure}) part, that is isomorphic to $\mathbb{R}^{7}$:%
\begin{equation}
\mathbb{O}\cong \mathbb{R}\oplus \mathbb{R}^{7}
\end{equation}%
Correspondingly, given an octonion $a\in \mathbb{O}$, we can uniquely write 
\begin{equation*}
a=\func{Re}a+\func{Im}a
\end{equation*}%
where $\func{Re}a\in \mathbb{R}$, and $\func{Im}a\in \mathbb{R}^{7}$. We can
now use octonion multiplication to define a vector cross product $\times $
on $\mathbb{R}^{7}$. Given $u,v\in \mathbb{R}^{7}$, we regard them as
octonions in $\func{Im}\mathbb{O}$, multiply them together using octonion
multiplication, and then project the result to $\func{Im}\mathbb{O}$ to
obtain a new vector in $\mathbb{R}^{7}$:%
\begin{equation}
u\times v=\func{Im}\left( uv\right) .  \label{octovp}
\end{equation}%
The subgroup of $GL\left( 7,\mathbb{R}\right) $ that preserves this vector
cross product is then precisely the group $G_{2}$. A detailed account of the
properties of the octonions and their relationship to exceptional Lie groups
is given by John Baez in \cite{BaezOcto}. The structure constants of the
vector cross product define a $3$-form on $\mathbb{R}^{7}$, hence $G_{2}$ is
alternatively defined as the subgroup of $GL\left( 7,\mathbb{R}\right) $
that preserves a particular $3$-form $\varphi _{0}$ \cite{Joycebook}.

\begin{definition}
Let $\left( e^{1},e^{2},...,e^{7}\right) $ be the standard basis for $\left( 
\mathbb{R}^{7}\right) ^{\ast }$, and denote $e^{i}\wedge e^{j}\wedge e^{k}$
by $e^{ijk} $. Then define $\varphi _{0}$ to be the $3$-form on $\mathbb{R}%
^{7}$ given by 
\begin{equation}
\varphi _{0}=e^{123}+e^{145}+e^{167}+e^{246}-e^{257}-e^{347}-e^{356}.
\label{phi0def}
\end{equation}%
Then $G_{2}$ is defined as the subgroup of $GL\left( 7,\mathbb{R}\right) $
that preserves $\varphi _{0}$.
\end{definition}

In general, given a $n$-dimensional manifold $M$, a $G$-structure on $M$ for
some Lie subgroup $G$ of $GL\left( n,\mathbb{R}\right) $ is a reduction of
the frame bundle $F$ over $M$ to a principal subbundle $P$ with fibre $G$. A 
$G_{2}$-structure is then a reduction of the frame bundle on a $7$%
-dimensional manifold $M$ to a $G_{2}$-principal subbundle. The obstructions
for the existence of a $G_{2}$-structure are purely topological. It
well-known \cite{FernandezGray,FriedrichNPG2,Gray-VCP} that a manifold
admits a $G_{2}$-structure if and only if the Stiefel-Whitney classes $w_{1}$
and $w_{2}$ both vanish.

It turns out that there is a $1$-$1$ correspondence between $G_{2}$%
-structures on a $7$-manifold and smooth $3$-forms $\varphi $ for which the $%
7$-form-valued bilinear form $B_{\varphi }$ as defined by (\ref{Bphi}) is
positive definite (for more details, see \cite{Bryant-1987} and the arXiv
version of \cite{Hitchin:2000jd}). 
\begin{equation}
B_{\varphi }\left( u,v\right) =\frac{1}{6}\left( u\lrcorner \varphi \right)
\wedge \left( v\lrcorner \varphi \right) \wedge \varphi  \label{Bphi}
\end{equation}%
Here the symbol $\lrcorner $ denotes contraction of a vector with the
differential form, which can be written in local coordinates as 
\begin{equation}
\left( u\lrcorner \varphi \right) _{mn}=u^{a}\varphi _{amn}
\label{contractdef}
\end{equation}%
where we have also used the Einstein summation convention, which we will be
using henceforth whenever dealing with expressions in local coordinates.

A smooth $3$-form $\varphi $ is said to be \emph{positive }if $B_{\varphi }$
is the tensor product of a positive-definite bilinear form and a
nowhere-vanishing $7$-form. In this case, it defines a unique Riemannian
metric $g_{\varphi }$ and volume form $\mathrm{vol}_{\varphi }$ such that
for vectors $u$ and $v$, the following holds 
\begin{equation}
g_{\varphi }\left( u,v\right) \mathrm{vol}_{\varphi }=\frac{1}{6}\left(
u\lrcorner \varphi \right) \wedge \left( v\lrcorner \varphi \right) \wedge
\varphi  \label{gphi}
\end{equation}%
An equivalent way of defining a positive $3$-form $\varphi $, is to say that
at every point, $\varphi $ is in the $GL\left( 7,\mathbb{R}\right) $-orbit
of $\varphi _{0}$. It can be easily checked that the metric (\ref{gphi}) for 
$\varphi =\varphi _{0}$ is in fact precisely the standard Euclidean metric $%
g_{0}$ on $\mathbb{R}^{7}$. Therefore, every $\varphi $ that is in the $%
GL\left( 7,\mathbb{R}\right) $-orbit of $\varphi _{0}$ has an \emph{%
associated} Riemannian metric $g$ that is in the $GL\left( 7,\mathbb{R}%
\right) $-orbit of $g_{0}.$ The only difference is that the stabilizer of $%
g_{0}$ (along with orientation) in this orbit is the group $SO\left(
7\right) $, whereas the stabilizer of $\varphi _{0}$ is $G_{2}\subset
SO\left( 7\right) $. This shows that positive $3$-forms forms that
correspond to the same metric, i.e., are \emph{isometric}, are parametrized
by $SO\left( 7\right) /G_{2}\cong \mathbb{RP}^{7}\cong S^{7}/\mathbb{Z}_{2}$%
. Therefore, on a Riemannian manifold, metric-compatible $G_{2}$-structures
are parametrized by sections of an $\mathbb{RP}^{7}$-bundle, or
alternatively, by sections of an $S^{7}$-bundle, with antipodal points
identified.

The \emph{intrinsic torsion }of a $G_{2}$-structure is defined by $\nabla
\varphi $, where $\nabla $ is the Levi-Civita connection for the metric $g$
that is defined by $\varphi $. Following \cite{karigiannis-2007}, we have 
\begin{subequations}
\begin{eqnarray}
\nabla _{a}\varphi _{bcd} &=&2T_{a}^{\ e}\psi _{ebcd}^{{}}  \label{codiffphi}
\\
\nabla _{a}\psi _{bcde} &=&-8T_{a[b}\varphi _{cde]}  \label{psitorsion}
\end{eqnarray}%
\end{subequations}%
where $T_{ab}$ is the \emph{full torsion tensor}, note that an additional
factor of $2$ is for convenience, and $\psi =\ast \varphi $ is the $4$-form
that is the Hodge dual of $\varphi $ with respect to the metric $g$. In
general we can split $T_{ab}$ according to irreducible representations $%
\mathbf{1},$ $\mathbf{7},$ $\mathbf{14},$ and $\mathbf{27}$ of $G_{2}$ into 
\emph{torsion components}: 
\begin{equation}
2T=\frac{1}{4}\tau _{0}g-\tau _{1}\lrcorner \varphi +\frac{1}{2}\tau _{2}-%
\frac{1}{3}\tau _{3}  \label{Tdecomp}
\end{equation}%
where $\tau _{0}$ is a function, and gives the $\mathbf{1}$ component of $T$%
. We also have $\tau _{1}$, which is a $1$-form and hence gives the $\mathbf{%
7}$ component, $\tau _{2}\ $is a $2$-form in the $\mathbf{14}$
representation, and $\tau _{3}$ is a traceless symmetric $2$-tensor, giving the $\mathbf{%
27}$ component. As shown by Karigiannis in \cite{karigiannis-2007}, the
torsion components $\tau _{i}$ relate directly to the expression for $%
d\varphi $ and $d\psi $. In fact, in our notation, 
\begin{subequations}%
\label{dptors} 
\begin{eqnarray}
d\varphi &=&\tau _{0}\psi +3\tau _{1}\wedge \varphi +\ast \mathrm{i}%
_{\varphi }\left( \tau _{3}\right)  \label{dphi} \\
d\psi &=&4\tau _{1}\wedge \psi +\ast \tau _{2}.  \label{dpsi}
\end{eqnarray}%
\end{subequations}%
Here $\mathrm{i}_{\varphi }$ is a map that takes symmetric $2$-tensors to $3$%
-forms and given a decomposable $2$-tensor $\alpha \otimes \alpha $, where $%
\alpha $ is a $1$-form,
\begin{equation*}
\mathrm{i}_{\varphi }\left( \alpha \otimes \alpha \right) =\frac{1}{3}\alpha
\wedge \left( \alpha \lrcorner \varphi \right) .
\end{equation*}%
Note that in \cite%
{GrigorianCoflow,GrigorianG2Torsion1,GrigorianSU3flow,GrigorianYau1} a
different convention for the torsion is used is used: $\tau _{1}$ in that
convention corresponds to $\frac{1}{4}\tau _{0}$ here, $\tau _{7}$
corresponds to $-\tau _{1}$ here,\ $\mathrm{i}_{\varphi }\left( \tau
_{27}\right) $ corresponds to $-\frac{1}{3}\tau _{3},$ and $\tau _{14}$
corresponds to $\frac{1}{2}\tau _{2}$. The notation used here is widely used
elsewhere in the literature.

An important special case is when the $G_{2}$-structure is torsion-free,
that is, $T=0$. This is equivalent to $\nabla \varphi =0$, and hence
torsion-free $G_{2}$-structures are also called parallel $G_{2}$-structures.
Also, by Fern\'{a}ndez and Gray \cite{FernandezGray}, this condition is also
equivalent to $d\varphi =d\psi =0$. Moreover, a $G_{2}$-structure is
torsion-free if and only if the holonomy of the corresponding metric is
contained in $G_{2}$ \cite{Joycebook}. On a compact manifold, the holonomy
group is then precisely equal to $G_{2}$ if and only if the fundamental
group $\pi _{1}$ is finite. If $d\varphi =0$, then we say $\varphi $ defines
a \emph{closed }$G_{2}$-structure. In that case, $\tau _{0}=\tau _{1}=\tau
_{3}=0$ and only $\tau _{2}$ is in general non-zero. In this case, $T=\frac{1%
}{4}\tau _{2}$ and is hence skew-symmetric. If instead, $d\psi =0$, then we
say that we have a \emph{co-closed }$G_{2}$-structure. In this case, $\tau
_{1}$ and $\tau _{2}$ vanish in (\ref{dpsi}) and we are left with $\tau _{0}$
and $\tau _{3}$ components. In particular, the torsion tensor $T_{ab}$ is
now symmetric.\ There are of course other, intermediate, torsion classes.
For example, if $\tau _{1}$ is the only non-zero torsion component, the $%
G_{2}$ structure is said to be locally conformally parallel, since it is
known \cite{CleytonIvanovConf,GrigorianG2Torsion1} that a conformal
transformation can at least locally give a parallel $G_{2}$-structure. If $%
\tau _{1}$ is exact, then a suitable conformal transformation gives a global
parallel $G_{2}$-structure.

\section{Octonion bundle}

\setcounter{equation}{0}\label{secOctoCovDiv}In \cite{GrigorianOctobundle},
the author defined the octonion bundle on a manifold with a $G_{2}$%
-structure.

\begin{definition}
Let $M$ be a smooth $7$-dimensional manifold with a $G_{2}$-structure $%
\left( \varphi ,g\right) $. The \emph{octonion bundle }$\mathbb{O}M\cong
\Lambda ^{0}\oplus TM$ on $M$ is a rank $8$ real vector bundle equipped with
an octonion product of sections given by 
\begin{equation}
A\circ _{\varphi }B=\left( 
\begin{array}{c}
ab-g\left( \alpha ,\beta \right)  \\ 
a\beta +b\alpha +\alpha \times _{\varphi }\beta 
\end{array}%
\right)   \label{octoproddef}
\end{equation}%
for any sections $A=\left( a,\alpha \right) $ and $B=\left( b,\beta \right) $%
. Here we define $\times _{\varphi }$ by $g\left( \alpha \times _{\varphi
}\beta ,\gamma \right) =\varphi \left( \alpha ,\beta ,\gamma \right) $ and
given $A\in \Gamma \left( \mathbb{O}M\right) $, we write $A=\left( \func{Re}%
A,\func{Im}A\right) .$ The metric on $TM$ is extended to $\mathbb{O}M$ to
give the octonion inner product $\left\langle A,B\right\rangle =ab+g\left(
\alpha ,\beta \right) $.
\end{definition}

The product (\ref{octoproddef}) is non-associative and the associator for $%
\circ _{\varphi }$ is given by 
\begin{eqnarray}
\left[ A,B,C\right] _{\varphi } &=&A\circ _{\varphi }\left( B\circ _{\varphi
}C\right) -\left( A\circ _{\varphi }B\right) \circ _{\varphi }C
\label{assoc} \\
&=&2\left( \psi \left( \cdot ,\alpha ,\beta ,\gamma \right) \right) ^{\sharp
}  \notag
\end{eqnarray}%
where $\alpha ,\beta ,\gamma $ are the imaginary parts of $A,B,C$ and $%
\left( \psi \left( \cdot ,\alpha ,\beta ,\gamma \right) \right) ^{\sharp }$
is the vector field obtained from the $1$-form $\psi \left( \cdot ,\alpha
,\beta ,\gamma \right) $ using the metric.

Given the octonion bundle $\mathbb{O}M$ with the octonion algebra defined by
the $G_{2}$-structure $\varphi $ with torsion tensor $T$, we can extend the
Levi-Civita connection $\nabla $ to sections of $\mathbb{O}M$. Let $A=\left(
a,\alpha \right) \in \Gamma \left( \mathbb{O}M\right) ,$ then define the
covariant derivative on $\mathbb{O}M$ as 
\begin{equation}
\nabla _{X}A=\left( \nabla _{X}a,\nabla _{X}\alpha \right)  \label{delXA}
\end{equation}%
for any $X\in \Gamma \left( TM\right) $. Then, as shown in \cite%
{GrigorianOctobundle} 
\begin{equation}
\nabla _{X}\left( AB\right) =\left( \nabla _{X}A\right) \circ _{\varphi
}B+A\circ _{\varphi }\left( \nabla _{X}B\right) -\left[ T_{X},A,B\right]
\label{nablaXAB}
\end{equation}%
where $T_{X}=\left( 0,X\lrcorner T\right) $. We can then define an adapted
octonion covariant derivative.

\begin{definition}
Define the octonion covariant derivative $D$ such for any $X\in \Gamma
\left( TM\right) ,$%
\begin{equation*}
D_{X}:\Gamma \left( \mathbb{O}M\right) \longrightarrow \Gamma \left( \mathbb{%
O}M\right)
\end{equation*}%
given by 
\begin{equation}
D_{X}A=\nabla _{X}A-A\circ _{\varphi }T_{X}  \label{octocov}
\end{equation}%
for any $A\in \Gamma \left( \mathbb{O}M\right) .$ As before, $T_{X}=\left(
0,X\lrcorner T\right) \in \Gamma \left( \func{Im}\mathbb{O}M\right) $.
\end{definition}

From now on, let us suppress $\circ _{\varphi }$ for octonion product
defined by $\varphi $. As shown in \cite{GrigorianOctobundle}, $D$ satisfies
a number of useful properties. In particular, it is metric-compatible, and
satisfies a partial product rule%
\begin{equation}
D_{X}\left( AB\right) =\left( \nabla _{X}A\right) B+A\left( D_{X}B\right) .
\label{octocovprod}
\end{equation}%
We can also see that%
\begin{equation}
D_{X}1=-T_{X}.  \label{DX1}
\end{equation}%
For a fixed vector field $X$, we have $T_{X}=\left( 0,X\lrcorner T\right)
\in \Gamma \left( \func{Im}\mathbb{O}M\right) $, so the full torsion tensor $%
T$ may now be interpreted as a $1$-form with values in $\func{Im}\mathbb{O}M$%
, that is, $T$ is a map from $\Gamma \left( TM\right) \ \ $to $\Gamma \left( 
\func{Im}\mathbb{O}M\right) $ that takes $X$ to $T_{X}$. So as in \cite%
{GrigorianOctobundle}, we will regard $T\in \Omega ^{1}\left( \func{Im}%
\mathbb{O}M\right) $.

Recall from \cite{GrigorianOctobundle}, that given a unit octonion section $V
$ on $\mathbb{O}M$ we may define a modified product on $\mathbb{O}M$:%
\begin{equation}
A\circ _{V}B=\left( AV\right) \left( V^{-1}B\right) =AB+\left[ A,B,V\right]
V^{-1}  \label{OctoVAB}
\end{equation}%
This product then induces a new $G_{2}$-structure that is compatible with
the same metric $g$ as $\varphi $ and is given by 
\begin{equation}
\sigma _{V}\left( \varphi \right) =\left( v_{0}^{2}-\left\vert v\right\vert
^{2}\right) \varphi -2v_{0}v\lrcorner \psi +2v\wedge \left( v\lrcorner
\varphi \right)   \label{sigmaAdef}
\end{equation}%
where $V=\left( v_{0},v\right) $. It was explained by Bryant in \cite%
{bryant-2003} that all $G_{2}$-structures that are isometric to $\varphi $
are given by (\ref{sigmaAdef}) for some $V$. In particular, this also gives
an explicit parametrization of $G_{2}$-structures that are compatible with $g
$ as sections of an $S^{7}/\mathbb{Z}_{2}\cong \mathbb{R}\mathbb{P}^{7}$%
-bundle over $M$. In \cite{GrigorianOctobundle} it was shown that given two
unit octonion sections $U$ and $V$, 
\begin{equation}
\sigma _{U}\left( \sigma _{V}\left( \varphi \right) \right) =\sigma
_{UV}\left( \varphi \right) .  \label{sigmaprod}
\end{equation}%
This allows to move easily between isometric $G_{2}$-structures. Moreover,
it was also shown how the torsion and hence the octonion covariant
derivative $D$ depend on the choice of $V.$

\begin{theorem}[\protect\cite{GrigorianOctobundle}]
\label{ThmTorsV}Let $M$ be a smooth $7$-dimensional manifold with a $G_{2}$%
-structure $\left( \varphi ,g\right) $ with torsion $T\in \Omega ^{1}\left( 
\func{Im}\mathbb{O}M\right) $ and corresponding octonion covariant
derivative $D$. For a unit section $V\in \Gamma \left( \mathbb{O}M\right) ,$
consider the $G_{2}$-structure $\sigma _{V}\left( \varphi \right) .$ Then,
the torsion $T^{\left( V\right) }$ of $\sigma _{V}\left( \varphi \right) $
is given by 
\begin{equation}
T^{\left( V\right) }=-\left( DV\right) V^{-1}.  \label{TorsV3}
\end{equation}%
Also, let $D^{\left( V\right) }$ be the octonion covariant derivative
corresponding to $\sigma _{V}\left( \varphi \right) .$ Then, for any
octonion section $A$, we have, 
\begin{equation}
D^{\left( V\right) }A=\left( D\left( AV\right) \right) V^{-1}.
\label{DtildeAV2}
\end{equation}
\end{theorem}

We will refer to a particular choice of a $G_{2}$-structure on $M$ as a 
\emph{background }$G_{2}$-structure. Namely, given a background $G_{2}$%
-structure $\varphi $, we will write any other isometric $G_{2}$-structure
as $\sigma _{V}\left( \varphi \right) ,$ or will just refer to it as the $%
G_{2}$-structure defined by the octonion section $V$. Similarly, the
octonion derivative $D$ will be defined relative to $\varphi $ and its
torsion $T$. From (\ref{sigmaprod}) and (\ref{DtildeAV2}) we see that we can
easily change the background $G_{2}$-structure.

For some tensor bundle $\mathcal{T}$ on $M,$ define $\mathcal{T}\otimes 
\mathbb{O}M$ to be the bundle of octonion-valued tensors. Then we can extend 
$D$ to sections of $\mathcal{T}\otimes \mathbb{O}M$, and in particular we
can also define the covariant exterior derivative on sections $\Omega
^{p}\left( \mathbb{O}M\right) $ of the bundle of octonion-valued
differential forms $\left( \Lambda ^{p}T^{\ast }M\right) \otimes \mathbb{O}M$
\begin{equation}
d_{D}:\Omega ^{p}\left( \mathbb{O}M\right) \longrightarrow \Omega
^{p+1}\left( \mathbb{O}M\right) .  \label{Dextdiff}
\end{equation}%
such that 
\begin{equation}
d_{D}Q=d_{\nabla }Q-\left( -1\right) ^{p}Q\overset{\circ }{\wedge }T
\label{Dextdiff2}
\end{equation}%
where $d_{\nabla }$ is the skew-symmetrized $\nabla $ and $\overset{\circ }{%
\wedge }$ is a combination of exterior product and octonion product. Also
define the \emph{divergence }of a $p$-form \ $P$ with respect to $D$ as the $%
\left( p-1\right) $-form $\func{Div}P,$ given by 
\begin{equation}
\left( \func{Div}P\right) _{b_{2}...b_{p}}=D_{b_{1}}P_{\
b_{2}..b_{p}}^{b_{1}}.  \label{DivPdef}
\end{equation}%
In \cite{GrigorianOctobundle} we found the following properties of $T$ as a $%
\func{Im}\mathbb{O}M$-valued $1$-form

\begin{proposition}
\label{PropTorsDeriv} Suppose the octonion product on $\mathbb{O}M$ is
defined by the $G_{2}$-structure $\varphi $ with torsion $T$. Then, 
\begin{eqnarray}
d_{D}T &=&\frac{1}{4}\left( \pi _{7}\func{Riem}\right)  \label{extDT} \\
\func{Div}T &=&\left\vert T\right\vert ^{2}+\func{div}T  \label{dsT}
\end{eqnarray}%
where $\pi _{7}\func{Riem}\in \Omega ^{2}\left( \func{Im}\mathbb{O}M\right) $
$\cong \Omega ^{2}\left( TM\right) $- a vector-valued $2$-form given by $%
\left( \pi _{7}\func{Riem}\right) _{ab}^{\ \ c}=\left( \func{Riem}\right)
_{abmn}\varphi ^{mnc}.$ Also, $\func{div}T\in \Omega ^{0}\left( \func{Im}%
\mathbb{O}M\right) $ is given by\ $\left( \func{div}T\right) ^{a}=\nabla
^{b}T_{b}^{\ a}$ and $\left\vert T\right\vert ^{2}\in \Omega ^{0}\left( 
\func{Re}\mathbb{O}M\right) $ is given by $\left\vert T\right\vert
^{2}=T_{ab}T^{ab}$.
\end{proposition}

In particular, using Proposition \ref{PropTorsDeriv}, we can now work out
the commutator $\left[ D_{a},D_{b}\right] $ on octonion-valued tensors.

\begin{lemma}
\label{propdD2}Suppose $P\in \Gamma \left( \mathcal{T}\otimes \mathbb{O}%
M\right) $. Then, 
\begin{equation}
D_{a}D_{b}P-D_{b}D_{a}P=\func{Riem}\left( P\right) _{ab}-\frac{1}{4}P\left(
\pi _{7}\func{Riem}\right) _{ab}  \label{d2D}
\end{equation}%
where $\func{Riem}\left( P\right) $ gives the action of the Riemann
curvature endomorphism on $P$ regarded as a section of $\mathcal{T\oplus }%
\left( \mathcal{T}\otimes TM\right) .$
\end{lemma}

\begin{proof}
From the definition of $D$ (\ref{octocov}) as well as the product rule
property (\ref{octocovprod}), we have 
\begin{eqnarray*}
D_{a}D_{b}P &=&D_{a}\left( \nabla _{b}P-PT_{b}\right) \\
&=&\nabla _{a}\nabla _{b}P-\left( \nabla _{b}P\right) T_{a}-\left( \nabla
_{a}P\right) T_{b}-P\left( D_{a}T_{b}\right)
\end{eqnarray*}%
and hence, 
\begin{eqnarray*}
D_{a}D_{b}P-D_{b}D_{a}P &=&\nabla _{a}\nabla _{b}P-\nabla _{b}\nabla
_{a}P-P\left( D_{a}T_{b}-D_{b}T_{a}\right) \\
&=&\func{Riem}\left( P\right) _{ab}-\frac{1}{4}P\left( \pi _{7}\func{Riem}%
\right) _{ab}
\end{eqnarray*}%
where we have also used (\ref{extDT}).
\end{proof}

For convenience, we'll denote the curvature operator by $F$, so that 
\begin{equation}
F_{ab}\left( P\right) =\func{Riem}\left( P\right) _{ab}-\frac{1}{4}P\left(
\pi _{7}\func{Riem}\right) _{ab}  \label{Fabdef}
\end{equation}

Define the Laplacian operator $\Delta _{D}$ on $\mathbb{O}M$-valued tensors
as 
\begin{equation}
\Delta _{D}P=D^{a}D_{a}P  \label{lapD}
\end{equation}%
where $P\in \Gamma \left( \mathcal{T}\otimes \mathbb{O}M\right) .$ More
explicitly, this is given by 
\begin{eqnarray}
\Delta _{D}P &=&D^{a}\left( D_{a}P\right)  \notag \\
&=&D^{a}\left( \nabla _{a}P-PT_{a}\right)  \notag \\
&=&D^{a}\left( \nabla _{a}P\right) -\left( \nabla ^{a}P\right) T_{a}-P\left(
D^{a}T_{a}\right)  \notag \\
&=&\Delta P-2\left( \nabla _{a}P\right) T^{a}-P\left( \func{Div}T\right)
\label{lapD2}
\end{eqnarray}%
For a tensor product of two $\mathbb{O}M$-valued tensors, we find%
\begin{eqnarray}
\Delta _{D}\left( P\otimes Q\right) &=&D^{a}\left( \left( \nabla
_{a}P\right) \otimes Q+P\otimes \left( D_{a}Q\right) \right)  \notag \\
&=&\left( \Delta P\right) \otimes Q+2\left( \nabla _{a}P\right) \otimes
\left( D^{a}Q\right) +P\otimes \left( \Delta _{D}Q\right)  \label{lapD3}
\end{eqnarray}

We will also need to know how to commute $\Delta _{D}$ and $D$.

\begin{lemma}
\label{lemLaplComm}Suppose $P\in \Gamma \left( \mathcal{T}\otimes \mathbb{O}%
M\right) $. Then,%
\begin{eqnarray}
D_{b}\left( \Delta _{D}P\right) -\Delta _{D}\left( D_{b}P\right) &=&-2\left( 
\func{Riem}_{ab}\nabla ^{a}\right) \left( P\right) +\frac{1}{4}\left( \nabla
^{a}P\right) \left( \pi _{7}\func{Riem}\right) _{ab}  \label{DLapD} \\
&&-\func{Ric}_{bc}\nabla ^{c}P-\func{Riem}_{b}^{\ a}\left( PT_{a}\right) +%
\func{Riem}_{ab}\left( P\right) T^{a}  \notag \\
&&+\frac{1}{4}P\left( D^{a}\left( \pi _{7}\func{Riem}\right) _{ab}\right)
-\left( \func{div}\func{Riem}\right) _{b}\left( P\right) .  \notag
\end{eqnarray}%
where $\func{Riem}$ is the Riemann curvature endomorphism on an appropriate
tensor bundle.
\end{lemma}

\begin{proof}
Using (\ref{d2D}) and (\ref{octocovprod}) repeatedly, we have

\begin{eqnarray*}
D_{b}\left( \Delta _{D}P\right) 
&=&D_{b}D^{a}D_{a}P=D^{a}D_{b}D_{a}P+F_{b}^{\ a}\left( D_{a}P\right)  \\
&=&\Delta _{D}D_{b}P-D^{a}\left( F_{ab}\left( P\right) \right) +F_{b}^{\
a}\left( D_{a}P\right) 
\end{eqnarray*}%
More concretely, 
\begin{eqnarray}
D^{a}\left( F_{ab}\left( P\right) \right)  &=&D^{a}\left( \func{Riem}%
_{ab}\left( P\right) -\frac{1}{4}P\left( \pi _{7}\func{Riem}\right)
_{ab}\right)   \notag \\
&=&\left( \nabla ^{a}\func{Riem}_{ab}\right) \left( P\right) +\left( \func{%
Riem}_{ab}\nabla ^{a}\right) \left( P\right) -\func{Riem}_{ab}\left(
P\right) T^{a}  \notag \\
&&-\frac{1}{4}\left( \nabla ^{a}P\right) \left( \pi _{7}\func{Riem}\right)
_{ab}-\frac{1}{4}P\left( D^{a}\left( \pi _{7}\func{Riem}\right) _{ab}\right) 
\label{DivF} \\
F_{b}^{\ a}\left( D_{a}P\right)  &=&-\func{Riem}\left( \nabla _{a}P\right)
_{\ b}^{a}-\func{Riem}_{b}^{\ a}\left( PT_{a}\right)   \label{FD}
\end{eqnarray}%
We also have 
\begin{equation*}
\func{Riem}\left( \nabla _{a}P\right) _{\ b}^{a}=\func{Ric}_{bc}\nabla
^{c}P+\left( \func{Riem}_{\ b}^{a}\nabla _{a}\right) P
\end{equation*}%
where $\left( \func{Riem}_{\ b}^{a}\nabla _{a}\right) P$ means a composition
of operators $\nabla $ and $\func{Riem},$ both acting on sections of the
bundle $\mathcal{T\oplus }\left( \mathcal{T}\otimes TM\right) $, as opposed
to $\func{Riem}\left( \nabla _{a}P\right) _{\ b}^{a},$ where $\func{Riem}$
acts on $\nabla P$ as a section of the bundle $T^{\ast }M\otimes \left( 
\mathcal{T\oplus }\left( \mathcal{T}\otimes TM\right) \right) $. Combining
everything, we obtain (\ref{DLapD}).
\end{proof}

In (\ref{DLapD}), note that 
\begin{eqnarray}
D^{a}\left( \pi _{7}\func{Riem}\right) _{ab} &=&\nabla ^{a}\left( \func{Riem}%
_{abcd}\varphi ^{cdm}\delta _{m}\right) -\left( \pi _{7}\func{Riem}\right)
_{ab}T^{a}  \notag \\
&=&\left( \func{div}\func{Riem}\right) _{bcd}\varphi ^{cdm}\delta _{m}+2%
\func{Riem}_{abcd}T^{ae}\psi _{e}^{\ cdm}\delta _{m}  \notag \\
&&-\left( \pi _{7}\func{Riem}\right) _{ab}T^{a}  \label{Divpi7riem}
\end{eqnarray}%
where $\delta $ is the canonical $\func{Im}\mathbb{O}M$-valued $1$-form that
gives the isomorphism from $TM$ to $\func{Im}\mathbb{O}M,$ so in local
coordinates, for any value of the index $m$, $\delta _{m}$ is an imaginary
octonion. We see that any terms in (\ref{DLapD}) that do not involve
derivatives of $P$, either involve $\func{div}\func{Riem}$ or a combination
of $\func{Riem}$ and $T.$ Hence, we can schematically write 
\begin{equation}
D\left( \Delta _{D}P\right) =\Delta _{D}\left( DP\right) +\func{Riem}\ast
DP+\left( \func{div}\func{Riem}+\func{Riem}\ast T\right) \ast P
\label{DLapcomm}
\end{equation}%
where $\ast $ denotes some contraction involving $g$ and/or $\varphi $.

Consider $\left\langle \Delta _{D}P,P\right\rangle :$%
\begin{eqnarray*}
\left\langle \Delta _{D}P,P\right\rangle &=&\left\langle
D_{a}D^{a}P,P\right\rangle \\
&=&\nabla _{a}\left\langle D^{a}P,P\right\rangle -\left\vert DP\right\vert
^{2} \\
&=&\nabla _{a}\left( \nabla ^{a}\left\vert P\right\vert ^{2}-\left\langle
P,D^{a}P\right\rangle \right) -\left\vert DP\right\vert ^{2}
\end{eqnarray*}%
where we have used metric compatibility of $D$. Thus, 
\begin{equation}
\func{div}\left\langle DP,P\right\rangle =\frac{1}{2}\Delta \left\vert
P\right\vert ^{2}  \label{nablaDPP}
\end{equation}%
and hence, 
\begin{equation}
\left\langle \Delta _{D}P,P\right\rangle =\frac{1}{2}\Delta \left\vert
P\right\vert ^{2}-\left\vert DP\right\vert ^{2}.  \label{lapDPP}
\end{equation}%
In particular, for a unit octonion section $V,$ 
\begin{equation}
\left\langle \Delta _{D}V,V\right\rangle =-\left\vert DV\right\vert ^{2}.
\label{lapDVV}
\end{equation}

\section{Energy functional}

\setcounter{equation}{0}\label{secEnergy}Given a $7$-dimensional Riemannian
manifold that admit $G_{2}$-structures, we have a choice of $G_{2}$%
-structures that correspond to the given Riemannian metric $g$. As we have
seen, after fixing an arbitrary $G_{2}$-structure $\varphi $ in this metric
class, all the other $G_{2}$-structures that are compatible with $g$ are
parametrized by unit octonion sections, up to a sign. Given a unit octonion
section $V,$ the corresponding $G_{2}$-structure $\sigma _{V}\left( \varphi
\right) $ will have torsion $T^{\left( V\right) }$ given by $T^{\left(
V\right) }=-\left( DV\right) V^{-1},$ where $D$ is the octonion covariant
derivative with respect to $\varphi $. The question is how to pick the
\textquotedblleft best\textquotedblright\ representative of this metric
class. The choice of a particular $G_{2}$-structure in a fixed metric class
is similar to choosing a gauge in gauge theory. Obviously, if the metric has
holonomy contained in $G_{2},$ then the \textquotedblleft
best\textquotedblright\ representative should be a torsion-free $G_{2}$%
-structure that corresponds to that metric. On compact manifolds, a
reasonable approach would be to pick a gauge that minimizes some functional.
The natural choice is the $L^{2}$-norm of the torsion. Suppose $M$ is now
compact, in \cite{GrigorianOctobundle} the author defined the functional $%
\mathcal{E}:\Gamma \left( S\mathbb{O}M\right) \longrightarrow \mathbb{R}$ ,
where $S\mathbb{O}M$ is the unit sphere subbundle, by 
\begin{eqnarray}
\mathcal{E}\left( V\right) &=&\int_{M}\left\vert T^{\left( V\right)
}\right\vert ^{2}\func{vol} \\
&=&\int_{M}\left\vert \left( DV\right) V^{-1}\right\vert ^{2}\func{vol} \\
&=&\int_{M}\left\vert DV\right\vert ^{2}\func{vol}.  \label{eV3}
\end{eqnarray}%
This is simply the energy functional for unit octonion sections. It should
be noted that $\mathcal{E}\left( V\right) $ is independent of the choice of
the background $G_{2}$-structure and thus really only depends on the $G_{2}$%
-structure $\sigma _{V}\left( \varphi \right) $. So it may equivalently be
considered as a functional on the space of $G_{2}$-structures that are
compatible with the metric $g$. A similar energy functional for spinors has
been studied by Ammann, Weiss and Witt \cite{AmmannWeissWitt1}, however in
their case, the metric was unconstrained, and so the functional was both on
spinors and metrics.

Using the properties of $D,$ we easily obtain the critical points.

\begin{proposition}[\protect\cite{GrigorianOctobundle}]
The critical points of $\mathcal{E}$ satisfy 
\begin{equation}
\Delta _{D}V+\left\vert DV\right\vert ^{2}V=0  \label{unitvecteq}
\end{equation}%
and equivalently 
\begin{equation}
\func{div}T^{\left( V\right) }=0.  \label{divT0}
\end{equation}
\end{proposition}

The condition (\ref{divT0}) comes from the identity 
\begin{equation}
\Delta _{D}V+\left\vert DV\right\vert ^{2}V=-\left( \func{div}T^{\left(
V\right) }\right) V.  \label{lapdivT}
\end{equation}

We see from (\ref{divT0}) that the critical points of $\mathcal{E}$
correspond to $G_{2}$-structures that have divergence-free torsion. This
description fits very well with the interpretation of the $G_{2}$-structure
torsion as a connection for a non-associative gauge theory. The condition $%
\func{div}T=0$ is then simply the analog of the Coulomb gauge. It is
well-known (e.g. \cite{DonaldsonKronheimer,DonaldsonGauge,TaoBook}) that in
gauge theory, given some reference connection $A_{0}$, a connection $%
A=A_{0}+a$ is said to be in the Coulomb gauge relative to $A_{0}$ if $%
d_{A_{0}}^{\ast }a=0$ and $A$ is gauge equivalent to $A_{0}$. Moreover, $a$
then corresponds to critical points of the $L^{2}$-norm of $A-A_{0}$ within
the gauge group orbit of $A_{0}$. In our situation, we have a very similar
thing happening, where the Levi-Civita connection $\nabla $ plays the role
of the reference connection $A_{0}$ and $T$ has the role of $a$. The
divergence-free torsion condition can equivalently be written as $d_{\nabla
}^{\ast }T=0$.

In general, unless $DV=0$ (and hence $T^{\left( V\right) }=0$), critical
points of $\mathcal{E}$ with $\func{div}T^{\left( V\right) }=0$ will not be
local extrema of $\mathcal{E}.$

\begin{proposition}
Suppose $V\left( s,t\right) $ is a two-parameter family of unit octonion
sections, then the Hessian of $\mathcal{E}$ at a critical point is given by 
\begin{equation}
\frac{\partial ^{2}\mathcal{E}\left( V\left( s,t\right) \right) }{\partial
s\partial t}=2\int_{M}\left( \left\langle D\dot{V},DV^{\prime }\right\rangle
-\left\vert DV\right\vert ^{2}\left\langle \dot{V},V^{\prime }\right\rangle
\right) \func{vol}  \label{EsecondVar}
\end{equation}%
where $\dot{V}=\frac{\partial }{\partial t}V\left( s,t\right) $ and $%
V^{\prime }=\frac{\partial }{\partial s}V\left( s,t\right) $.
\end{proposition}

\begin{proof}
To enforce the condition $\left\vert V\right\vert ^{2}=1,$ we may rewrite $%
\mathcal{E}$ as a functional on $\Gamma \left( \mathbb{O}M\right) $ with a
Lagrange multiplier $\lambda $:%
\begin{equation*}
\mathcal{E}\left( V\right) =\int_{M}\left( \left\vert DV\right\vert
^{2}-\lambda \left( \left\vert V\right\vert ^{2}-1\right) \right) \func{vol}
\end{equation*}%
where $\lambda =\left\vert DV\right\vert ^{2}$ at a critical point. From 
\cite{GrigorianOctobundle} we know that the first variation is given by 
\begin{eqnarray*}
\frac{\partial }{\partial t}\mathcal{E}\left( V\left( s,t\right) \right)
&=&\int_{M}\left( \frac{\partial }{\partial t}\left\vert DV\left( s,t\right)
\right\vert ^{2}-\lambda \frac{\partial }{\partial t}\left\vert V\left(
s,t\right) \right\vert ^{2}\right) \func{vol} \\
&=&2\int_{M}\left( \left\langle D\frac{\partial }{\partial t}V\left(
s,t\right) ,DV\left( s,t\right) \right\rangle -\lambda \left\langle V\left(
s,t\right) ,\frac{\partial }{\partial t}V\left( s,t\right) \right\rangle
\right) \func{vol}
\end{eqnarray*}%
and hence, the second variation is 
\begin{eqnarray*}
\frac{\partial ^{2}\mathcal{E}\left( V\left( s,t\right) \right) }{\partial
s\partial t} &=&2\int_{M}\left( \left\langle D\frac{\partial }{\partial t}%
V\left( s,t\right) ,D\frac{\partial }{\partial s}V\left( s,t\right)
\right\rangle +\left\langle D\frac{\partial ^{2}}{\partial s\partial t}%
V\left( s,t\right) ,DV\left( s,t\right) \right\rangle \right. \\
&&\left. -\lambda \left\langle \frac{\partial }{\partial s}V\left(
s,t\right) ,\frac{\partial }{\partial t}V\left( s,t\right) \right\rangle
-\lambda \left\langle V\left( s,t\right) ,\frac{\partial ^{2}}{\partial
s\partial t}V\left( s,t\right) \right\rangle \right) \func{vol} \\
&=&2\int_{M}\left( \left\langle D\frac{\partial }{\partial t}V\left(
s,t\right) ,D\frac{\partial }{\partial s}V\left( s,t\right) \right\rangle
-\lambda \left\langle \frac{\partial }{\partial s}V\left( s,t\right) ,\frac{%
\partial }{\partial t}V\left( s,t\right) \right\rangle \right. \\
&&\left. -\left\langle \frac{\partial ^{2}}{\partial s\partial t}V\left(
s,t\right) ,\Delta _{D}V\left( s,t\right) +\lambda V\left( s,t\right)
\right\rangle \right) \func{vol},
\end{eqnarray*}%
where we have integrated by parts. However, at a critical point $\lambda
=\left\vert DV\right\vert ^{2}$ and (\ref{unitvecteq}) is satisfied, hence
the second derivative term vanishes and at a critical point obtain (\ref%
{EsecondVar}).
\end{proof}

The characterization of divergence-free torsion as corresponding to critical
points of the functional $\mathcal{E}$ shows that $G_{2}$-structures with
such torsion are in some sense special. On the other hand, it is quite a
broad class of $G_{2}$-structures. In \cite{GrigorianOctobundle}, a Dirac
operator $\NEG{D}$ was defined on the octonion bundle. For an octonion
section $V$, nn local coordinates it is given by $\NEG{D}V=\delta ^{a}\circ
_{\varphi }\left( D_{a}V\right) $, where $\delta $ is the canonical $\func{Im%
}\mathbb{O}M$-valued $1$-form as defined in Section \ref{secOctoCovDiv} and $%
\circ _{\varphi }$ is the octonion product defined by the $G_{2}$-structure $%
\varphi $. This definition is analogous to the standard definition on
spinors using Clifford multiplication. It was then shown that unit norm
eigensections of $\NEG{D}$ are critical points of $\mathcal{E}.$ These
correspond to $G_{2}$-structures with torsion that have constant $\tau _{0}$
and vanishing $\tau _{1}$, but with arbitrary $\tau _{2}$ and $\tau _{3}$.
An almost complementary set of $G_{2}$-structures also yields
divergence-free torsion - these are locally conformally parallel $G_{2}$%
-structures with $\tau _{0}=\tau _{2}=\tau _{3}=0$ and $\tau _{1}\neq 0$.
Overall, we have the following.

\begin{theorem}
\label{thmDivT0}Suppose $\varphi $ is a $G_{2}$-structure on a $7$%
-dimensional manifold, with torsion $T$ and components of torsion $\tau
_{0},\tau _{1},\tau _{2},\tau _{3}$. Then, $\func{div}T=0$ if one of the
following holds:

\begin{enumerate}
\item $\tau _{0}$ is constant and $\tau _{1}=0$ and arbitrary $\tau _{2}$
and $\tau _{3}$

\item $\tau _{0}=\tau _{2}=\tau _{3}=0$ and arbitrary $\tau _{1}$
\end{enumerate}
\end{theorem}

\begin{proof}
The condition 1 is proved in \cite[Prop. 10.5]{GrigorianOctobundle}. For
condition 2, recall from \cite{FernandezGray}, that if $\tau _{0}=\tau
_{2}=\tau _{3}=0$, then $d\tau _{1}=0$. Then, from (\ref{Tdecomp}) and (\ref%
{codiffphi}), we have 
\begin{eqnarray*}
\left( \func{div}T\right) ^{b} &=&-\nabla _{a}\left( \tau _{1}^{c}\varphi
_{\ \ \ c}^{ab}\right) =-\tau _{1}^{c}\nabla _{a}\varphi _{\ \ \ c}^{ab} \\
&=&2\tau _{1}^{c}T_{ad}\psi _{\ \ \ \ \ c}^{adb}=-\tau _{1}^{c}\tau
_{1}^{e}\varphi _{ead}\psi _{\ \ \ \ \ c}^{adb} \\
&=&-4\tau _{1}^{c}\tau _{1}^{e}\varphi _{ce}^{\ \ \ b}=0.
\end{eqnarray*}
\end{proof}

\section{Heat flow}

\setcounter{equation}{0}\label{secHeat}In general, however, we don't know if
the functional $\mathcal{E}$ has any critical points for a given metric.
However, another approach, that has been successful in the study of harmonic
maps and in Yang-Mills theory is to consider the negative gradient
flow of $\mathcal{E}$. This gives the following initial value problem for a
time-dependent unit octonion section $V\left( t\right) \in \Gamma \left( S%
\mathbb{O}M\right) $:

\begin{equation}
\left\{ 
\begin{array}{c}
\frac{\partial V}{\partial t}=\Delta _{D}V+\left\vert DV\right\vert ^{2}V \\ 
V\left( 0\right) =V_{0},%
\end{array}%
\right.   \label{heatflow}
\end{equation}%
which was introduced in \cite{GrigorianOctobundle}. Here $D$ is defined with
respect to some background $G_{2}$-structure $\varphi $ with torsion $T$.
This will be unambiguous because time-dependent $G_{2}$-structure along the
flow will be denoted by $\varphi _{V}=\sigma _{V}\left( \varphi \right) \ $%
with torsion $T^{\left( V\right) }$ and Hodge dual $4$-form $\psi _{V}$.
Although initially we have to make a choice of background $\varphi ,$ we
find that the flow is actually invariant under a change of the background $%
G_{2}$-structure. Indeed, suppose $\tilde{\varphi}=\sigma _{U}\left( \varphi
\right) ,$ for some unit octonion section $U$, then from (\ref{sigmaprod}), 
\begin{equation}
\sigma _{V}\left( \varphi \right) =\sigma _{VU^{-1}}\left( \sigma _{U}\left(
\varphi \right) \right) =\sigma _{VU^{-1}}\left( \tilde{\varphi}\right) .
\end{equation}%
Moreover, from (\ref{DtildeAV2}), 
\begin{equation}
D^{\left( U\right) }\left( VU^{-1}\right) =\left( DV\right) U^{-1}
\label{DUVUinv}
\end{equation}%
where $D^{\left( U\right) }$ is the covariant derivative defined with
respect to $\tilde{\varphi}$. Now, consider the corresponding Laplacian $%
\Delta _{D^{\left( U\right) }}$: 
\begin{eqnarray}
\Delta _{D^{\left( U\right) }}\left( VU^{-1}\right)  &=&-\left( D^{\left(
U\right) }\right) ^{\ast }D^{\left( U\right) }\left( VU^{-1}\right)   \notag
\\
&=&-\left( D^{\left( U\right) }\right) ^{\ast }\left( \left( DV\right)
U^{-1}\right)   \notag \\
&=&\left( \Delta _{D}V\right) U^{-1}  \label{lapDUV}
\end{eqnarray}%
where we have applied (\ref{DUVUinv}) twice. Hence, if we set $W=VU^{-1}$,
we find that (\ref{heatflow}) is equivalent to 
\begin{equation}
\left\{ 
\begin{array}{c}
\frac{\partial W}{\partial t}=\Delta _{D^{\left( U\right) }}W+\left\vert
D^{\left( U\right) }W\right\vert ^{2}W \\ 
W\left( 0\right) =V_{0}U^{-1}%
\end{array}%
\right. .
\end{equation}%
Therefore, we can always change the background $G_{2}$-structure as
convenient.

The flow (\ref{heatflow}) is clearly parabolic and by standard parabolic
theory, therefore has short-time existence and uniqueness. In \cite%
{Bagaglini1,DGKisoflow}, this flow was reformulated explicitly in terms of the
imaginary part of $V$ and was explicitly shown to be parabolic as a PDE on
vector fields.

\begin{theorem}
There exists an $\varepsilon >0$ such that there exists a unique solution of
(\ref{heatflow}) on $M\times \lbrack 0,\varepsilon )$.
\end{theorem}

From (\ref{lapdivT}), an equivalent way of writing the flow (\ref{heatflow})
is 
\begin{equation}
\frac{\partial V}{\partial t}=-\left( \func{div}T^{\left( V\right) }\right)
V.  \label{DVdtdivT}
\end{equation}%
Moreover, as an evolution equation for $\varphi _{V}\left( t\right) =\sigma
_{V\left( t\right) }\left( \varphi \right) $, it can also be rewritten as 
\begin{equation}
\frac{\partial \varphi _{V}\left( t\right) }{\partial t}=2\left( \func{div}%
T\left( t\right) \right) \lrcorner \psi _{V}\left( t\right) 
\end{equation}%
which we can obtain from the following simple lemma.

\begin{lemma}
Suppose a one-parameter family of unit octonion sections $V\left( t\right) $
satisfies the evolution equation 
\begin{equation}
\frac{\partial V}{\partial t}=-QV  \label{VQflow}
\end{equation}%
for some time-dependent sections $Q\left( t\right) \in \Gamma \left( \func{Im%
}\mathbb{O}M\right) .$ Then, the corresponding $G_{2}$-structure $3$-forms $%
\varphi _{V}\left( t\right) =\sigma _{V\left( t\right) }\left( \varphi
\right) $ satisfy the evolution equation%
\begin{equation}
\frac{\partial \varphi _{V}\left( t\right) }{\partial t}=2Q\left( t\right)
\lrcorner \psi _{V}\left( t\right) .  \label{dphiQ}
\end{equation}%
and the torsion $T^{\left( V\right) }$ satisfies 
\begin{equation}
\frac{\partial T^{\left( V\right) }}{\partial t}=\nabla Q\left( t\right)
+2T^{\left( V\right) }\times _{V}Q\left( t\right)   \label{dTdt}
\end{equation}%
where $\times _{V}$ is the cross-product defined by the $G_{2}\,$-structure $%
\varphi _{V}\left( t\right) $.
\end{lemma}

\begin{proof}
We can extract $\frac{\partial \varphi _{V}\left( t\right) }{\partial t}$ by
considering what happens to the modified product $\circ _{V\left( t\right) }$%
 (\ref{OctoVAB}). Let $A$ and $B$ be two fixed octonions, then 
\begin{eqnarray*}
\frac{\partial }{\partial t}\left( A\circ _{V\left( t\right) }B\right)  &=&%
\frac{\partial }{\partial t}\left( \left( AV\right) \left( \bar{V}B\right)
\right)  \\
&=&\left( A\frac{\partial V}{\partial t}\right) \left( \bar{V}B\right)
+\left( AV\right) \left( \frac{\partial \bar{V}}{\partial t}B\right) 
\end{eqnarray*}%
where we have used $V^{-1}=\bar{V}$ since $V$ is a unit octonion. Using (\ref%
{VQflow}) and $\frac{\partial \bar{V}}{\partial t}=\bar{V}Q,$ we then obtain%
\begin{eqnarray*}
\frac{\partial }{\partial t}\left( A\circ _{V\left( t\right) }B\right) 
&=&-\left( A\left( QV\right) \right) \left( \bar{V}B\right) +\left(
AV\right) \left( \left( \bar{V}Q\right) B\right)  \\
&=&-\left( \left( A\circ _{V}Q\right) V\right) \left( \bar{V}B\right)
+\left( AV\right) \left( \bar{V}\left( Q\circ _{V}B\right) \right)  \\
&=&-\left( A\circ _{V}Q\right) \circ _{V}B+A\circ _{V}\left( Q\circ
_{V}B\right)  \\
&=&\left[ A,Q,B\right] _{V}
\end{eqnarray*}%
where we have again used the definition (\ref{OctoVAB}) of $\circ _{V}$ and $%
\left[ \cdot ,\cdot ,\cdot \right] _{V}$ is the associator with respect to $%
\circ _{V}$. Using the relationship (\ref{assoc}) between the associator and 
$\psi $, we obtain (\ref{dphiQ}).

Similarly, 
\begin{eqnarray}
\frac{\partial T^{\left( V\right) }}{\partial t} &=&-\frac{\partial \left(
\left( DV\right) V^{-1}\right) }{\partial t}=\left( D\left( QV\right)
\right) V^{-1}-\left( DV\right) \left( V^{-1}Q\right)  \notag \\
&=&D^{\left( V\right) }Q+T^{\left( V\right) }\circ _{V}Q  \label{DTV1} \\
&=&\nabla Q+2T^{\left( V\right) }\times _{V}Q.
\end{eqnarray}
\end{proof}

By definition of the negative gradient flow, the energy functional $\mathcal{%
E}$ is decreasing along the flow (\ref{heatflow}) whenever $\func{div}T\neq
0 $. More precisely, $\mathcal{E}\left( t\right) $ satisfies the following
equation, which follows immediately from (\ref{dTdt}) with $Q=\func{div}T$.

\begin{lemma}
\label{lemfuncevol}Along the flow (\ref{heatflow}), the functional $\mathcal{%
E}$ satisfies 
\begin{subequations}
\begin{eqnarray}
\frac{d\mathcal{E}}{dt} &=&-2\int_{M}\left\vert \func{div}T^{\left( V\right)
}\right\vert ^{2}\func{vol}  \label{DEdt} \\
\frac{d^{2}\mathcal{E}}{dt^{2}} &=&4\int_{M}\left( \left\vert D^{\left(
V\right) }\left( \func{div}T^{\left( V\right) }\right) \right\vert
^{2}-\left\vert T^{\left( V\right) }\right\vert ^{2}\left\vert \func{div}%
T^{\left( V\right) }\right\vert ^{2}\right) \func{vol}  \label{D2Edt}
\end{eqnarray}%
\end{subequations}%
where we regard $\func{div}T^{\left( V\right) }$ as sections of $\func{Im}%
\mathbb{O}M$ and $T^{\left( V\right) }\in \Omega ^{1}\left( \func{Im}\mathbb{%
O}M\right) .$ The norm $\left\vert \cdot \right\vert $ is obtained by
extending the metric to $\Omega ^{1}\left( \mathbb{O}M\right) $.
\end{lemma}

\begin{proof}
Using (\ref{dTdt}) with $Q=\func{div}T^{\left( V\right) }$, we have 
\begin{eqnarray*}
\frac{d\mathcal{E}}{dt} &=&2\int_{M}\left\langle T^{\left( V\right) },\frac{%
\partial T^{\left( V\right) }}{\partial t}\right\rangle \func{vol} \\
&=&2\int_{M}\left( \left\langle T^{\left( V\right) },\nabla Q\right\rangle
+2\left\langle T^{\left( V\right) },T^{\left( V\right) }\times
_{V}Q\right\rangle \right) \func{vol} \\
&=&-2\int_{M}\left\vert Q\right\vert ^{2}\func{vol}
\end{eqnarray*}%
where the second term in the second line vanishes by symmetry considerations
and the first term is integrated by parts.

Using (\ref{DTV1}), and suppressing $\circ _{V}$, we have 
\begin{eqnarray}
\frac{d^{2}\mathcal{E}}{dt^{2}} &=&-4\int_{M}\left\langle Q,\func{div}\left( 
\frac{\partial T^{\left( V\right) }}{\partial t}\right) \right\rangle \func{%
vol}  \notag \\
&=&-4\int_{M}\left\langle Q,\func{div}\left( D^{\left( V\right) }Q+T^{\left(
V\right) }Q\right) \right\rangle \func{vol}  \notag \\
&=&4\int_{M}\left\langle \nabla Q,D^{\left( V\right) }Q+T^{\left( V\right)
}Q\right\rangle \func{vol}  \notag \\
&=&4\int_{M}\left\langle D^{\left( V\right) }Q+QT^{\left( V\right)
},D^{\left( V\right) }Q+T^{\left( V\right) }Q\right\rangle \func{vol}  \notag
\\
&=&4\int_{M}\left( \left\vert D^{\left( V\right) }Q\right\vert
^{2}+\left\langle QT^{\left( V\right) }+T^{\left( V\right) }Q,D^{\left(
V\right) }Q\right\rangle \right.  \\
&&\left. +\left\langle QT^{\left( V\right) },T^{\left( V\right)
}Q\right\rangle \right) \func{vol}.  \notag
\end{eqnarray}%
Note that $Q$ and $T^{\left( V\right) }$ are both imaginary octonions, so $%
QT^{\left( V\right) }+T^{\left( V\right) }Q$ only has a real part. On the
other hand, in 
\begin{equation*}
D^{\left( V\right) }Q=\nabla Q-QT^{\left( V\right) }\text{,}
\end{equation*}%
the derivative term $\nabla Q$ is pure imaginary so the real part comes from 
$QT^{\left( V\right) }$. Hence%
\begin{equation*}
\left\langle QT^{\left( V\right) }+T^{\left( V\right) }Q,D^{\left( V\right)
}Q\right\rangle =-\left\langle QT^{\left( V\right) }+T^{\left( V\right)
}Q,QT^{\left( V\right) }\right\rangle .
\end{equation*}%
Thus, overall, 
\begin{equation*}
\frac{d^{2}\mathcal{E}}{dt^{2}}=4\int_{M}\left( \left\vert D^{\left(
V\right) }Q\right\vert ^{2}-\left\vert QT^{\left( V\right) }\right\vert
^{2}\right) \func{vol}.
\end{equation*}%
However, note that more explicitly, we can write 
\begin{eqnarray*}
\left\vert QT^{\left( V\right) }\right\vert ^{2} &=&g^{ab}\left\langle
QT_{a}^{\left( V\right) },QT_{b}^{\left( V\right) }\right\rangle  \\
&=&g^{ab}\left\langle \bar{Q}\left( QT_{a}^{\left( V\right) }\right)
,T_{b}^{\left( V\right) }\right\rangle =g^{ab}\left\langle \left( \bar{Q}%
Q\right) T_{a}^{\left( V\right) },T_{b}^{\left( V\right) }\right\rangle  \\
&=&\left\vert Q\right\vert ^{2}\left\vert T^{\left( V\right) }\right\vert
^{2}
\end{eqnarray*}%
and hence we obtain (\ref{D2Edt}).
\end{proof}

\begin{remark}
To work out the second derivative, we could alternatively use (\ref{dTdt})
to obtain 
\begin{eqnarray}
\frac{d^{2}\mathcal{E}}{dt^{2}} &=&-4\int_{M}\left\langle Q,\func{div}\left(
\nabla Q+2T^{\left( V\right) }\times _{V}Q\right) \right\rangle \func{vol} 
\notag \\
&=&4\int_{M}\left\vert \nabla Q\right\vert ^{2}+2\left\langle \nabla
Q,T^{\left( V\right) }\times _{V}Q\right\rangle \func{vol}.
\end{eqnarray}%
This is then essentially the same expression that appears in Lemma 5.10 in 
\cite{DGKisoflow}.\ 
\end{remark}

In \cite{DGKisoflow}, the second derivative of $\mathcal{E}$ was estimated
using the first non-zero eigenvalue of the Laplacian on vector fields as
long as the pointwise norm square of the torsion was sufficiently small.
From (\ref{D2Edt}), we can say that 
\begin{equation}
\frac{d^{2}\mathcal{E}}{dt^{2}}\geq 4\lambda _{1}\left( V\right)
\int_{M}\left\vert \func{div}T^{\left( V\right) }\right\vert ^{2}\func{vol}
\end{equation}%
where $\lambda _{1}\left( V\right) $ is lowest (non-zero) eigenvalue of the
operator $H_{V}=-\Delta _{D^{\left( V\right) }}-\left\vert T^{\left(
V\right) }\right\vert ^{2}.$ By compactness of $M$, this operator clearly
has a discrete spectrum. Also, from (\ref{DX1}) and (\ref{dsT}), we see that 
$\func{div}T^{\left( V\right) }=-H_{V}\left( 1\right) $, and hence, $\func{%
div}T^{\left( V\right) }$ is $L^{2}$-orthogonal to the kernel of $H_{V}$.
The operator $-\Delta _{D^{\left( V\right) }}$ has a non-negative spectrum
that is independent of $V,$ which can be seen from the covariance property (%
\ref{lapDUV}), however $H_{V}$ will in general have a spectrum that depends
on $V,$ and does not have to be non-negative. On the other hand, if $%
\left\vert T^{\left( V\right) }\right\vert ^{2}$ is less than first non-zero
eigenvalue of $-\Delta _{D^{\left( V\right) }}$, then $\lambda _{1}\left(
V\right) $ will be positive, and thus we obtain an analogue of the estimate
from \cite{DGKisoflow}.

\begin{corollary}
\label{corrDivTdecay}Let $\lambda >0$ be first non-zero eigenvalue of the
operator $-\Delta _{D}.$ Then, whenever $\left\vert T^{\left( V\right)
}\right\vert ^{2}=\left\vert DV\right\vert ^{2}\leq \frac{1}{2}\lambda $, 
\begin{equation}
\frac{d}{dt}\int_{M}\left\vert \func{div}T^{\left( V\right) }\right\vert ^{2}%
\func{vol}\leq -\lambda \int_{M}\left\vert \func{div}T^{\left( V\right)
}\right\vert ^{2}\func{vol}.
\end{equation}
\end{corollary}

We will adapt the techniques introduced by Shi for the Ricci flow \cite%
{ShiEstimate}, that were later used in \cite{GraysonHamilton} for the
harmonic map heat flow and in \cite{LotayWei1} for the Laplacian flow of
closed $G_{2}$-structures, to prove estimates for a finite time blow-up for
the flow (\ref{heatflow}). Let us introduce the quantity 
\begin{equation}
\Lambda \left( x,t\right) =\left\vert DV\left( x,t\right) \right\vert ^{2}.
\label{lambdaxt}
\end{equation}%
Of course, from (\ref{TorsV3}), we see that $\Lambda \left( x,t\right)
=\left\vert T^{\left( V\right) }\left( x,t\right) \right\vert ^{2}.$

At every $t\in \mathbb{R}$ for which (\ref{heatflow}) is defined, let us
also define 
\begin{equation}
\Lambda \left( t\right) =\sup_{x\in M}\Lambda \left( x,t\right) .
\label{lambdat}
\end{equation}%
Let $\Lambda \left( 0\right) =\Lambda _{0}$ be the maximal initial energy
density, and equivalently the maximal initial pointwise norm squared of the
torsion tensor $\sup_{x\in M}\left\vert T^{\left( V\right) }\left(
x,0\right) \right\vert $.

Our main result in this section is the following:

\begin{theorem}
\label{thmMain}Suppose $V\left( t\right) $ is a solution to (\ref{heatflow})
on a finite maximal time interval $[0,t_{\max })$. Then 
\begin{equation}
\lim_{t\longrightarrow t_{\max }^{-}}\Lambda \left( t\right) =\infty
\label{limLambdainf}
\end{equation}%
and moreover, 
\begin{equation}
\Lambda \left( t\right) \geq \frac{1}{2\left( t_{\max }-t\right) }-C_{0}
\label{lamdbatineq}
\end{equation}%
where $C_{0}>0$ depends on the curvature and torsion of the background $%
G_{2} $-structure.
\end{theorem}

The above theorem in particular shows that as long as $\Lambda \left(
t\right) ,$ and equivalently $\left\vert T^{\left( V\right) }\left(
x,t\right) \right\vert $, is bounded, a solution to (\ref{heatflow}) will
exist. To prove Theorem \ref{thmMain}, we will use the following strategy:

\begin{enumerate}
\item We will work out the evolution of $DV$ and hence $\left\vert
DV\right\vert ^{2}$ in Lemma \ref{lemDVevol}. From the Maximum Principle,
this will also give an upper bound for $\Lambda \left( t\right) $ in Theorem %
\ref{thmLambdaest}.

\item We will obtain the evolution of $\left\vert
D^{2}V\right\vert ^{2}$ and $\left\vert
D^{3}V\right\vert ^{2}$, and then in Theorem \ref{thmDkVest}, by induction we will obtain
bounds on $\left\vert D^{k}V\right\vert ^{2}$ in terms of $\Lambda .$

\item These bounds will then be used to show that whenever $\Lambda \left(
t\right) $ is finite, the flow $V\left( t\right) $ may be smoothly extended
further. This will then prove (\ref{limLambdainf}).
\end{enumerate}

In the estimates that follow, we will use $\ast $ to denote any multilinear
contraction that involves $g,g^{-1},\varphi ,\psi ,$ and we will drop
irrelevant constant factors. Sometimes we will generically use $C$ for a
constant, which may denote a different constant in different places.

\begin{lemma}
\label{lemDVevol}Along the flow (\ref{heatflow}), \ $\left\vert
DV\right\vert ^{2}\ $evolves as%
\begin{eqnarray}
\frac{\partial \left( \left\vert DV\right\vert ^{2}\right) }{\partial t}
&=&\Delta _{D}\left\vert DV\right\vert ^{2}-2\left\vert D^{2}V\right\vert
^{2}+2\left\vert DV\right\vert ^{4}-4\func{Riem}_{\ b\ n}^{a\ m}\left\langle
\left( \nabla _{a}v^{n}\right) \delta _{m},D^{b}V\right\rangle   \notag \\
&&-2\func{Ric}_{bc}\left\langle \nabla ^{b}V,D^{c}V\right\rangle +\frac{1}{2}%
\left\langle \left( \nabla ^{a}V\right) \left( \pi _{7}\func{Riem}\right)
_{ab},D^{b}V\right\rangle   \notag \\
&&-2\left( \func{div}\func{Riem}\right) _{b\ n}^{\ m}v^{n}\left\langle
\delta _{m},D^{b}V\right\rangle +\frac{1}{2}\left\langle V\func{Div}\left(
\pi _{7}\func{Riem}\right) _{b},D^{b}V\right\rangle   \notag \\
&&-2\left\langle \func{Riem}_{ab}\left( V\right) T^{a}-\func{Riem}_{b}^{\
a}\left( VT_{a}\right) ,D^{b}V\right\rangle   \label{ddvsqevol}
\end{eqnarray}%
where $v=\func{Im}V$. Moreover, the evolution of $\left\vert DV\right\vert
^{2}$ satisfies the following inequality 
\begin{equation}
\frac{\partial \left\vert DV\right\vert ^{2}}{\partial t}\leq \Delta
\left\vert DV\right\vert ^{2}-2\left\vert D^{2}V\right\vert ^{2}+2\left(
\left\vert DV\right\vert ^{4}+2R_{1}\left\vert DV\right\vert
^{2}+R_{2}\left\vert DV\right\vert \right)   \label{dDVsqdt}
\end{equation}%
where $R_{1}$ is a constant multiple of $\sup_{M}\left\vert \func{Riem}%
\right\vert $ and $R_{2}$ is a linear combination of $\sup_{M}\left\vert 
\func{div}\func{Riem}\right\vert $ and $\sup_{M}\left\vert T\right\vert
\left\vert \func{Riem}\right\vert $.
\end{lemma}

\begin{proof}
We have $V$ satisfying the flow 
\begin{equation}
\frac{\partial V}{\partial t}=\Delta _{D}V+\left\vert DV\right\vert ^{2}V
\end{equation}%
and hence, using Lemma \ref{lemLaplComm}, 
\begin{eqnarray}
\frac{\partial \left( DV\right) }{\partial t} &=&D\left( \frac{\partial V}{%
\partial t}\right) =D\left( \Delta _{D}V+\left\vert DV\right\vert
^{2}V\right)  \notag \\
&=&D\left( \Delta _{D}V\right) +\left( \nabla \left\vert DV\right\vert
^{2}\right) V+\left\vert DV\right\vert ^{2}DV  \notag \\
&=&\Delta _{D}\left( DV\right) +\func{Riem}\ast DV+\left( \func{div}\func{%
Riem}+\func{Riem}\ast T\right) \ast V  \label{dDVdt} \\
&&+\left( \nabla \left\vert DV\right\vert ^{2}\right) V+\left\vert
DV\right\vert ^{2}DV.  \notag
\end{eqnarray}%
Moreover, now, 
\begin{eqnarray}
\frac{\partial \left\vert DV\right\vert ^{2}}{\partial t} &=&2\left\langle 
\frac{\partial \left( DV\right) }{\partial t},DV\right\rangle  \notag \\
&=&2\left\langle D\left( \Delta _{D}V\right) ,DV\right\rangle +2\left\vert
DV\right\vert ^{4}.  \label{dDVsqnocurv}
\end{eqnarray}%
Using (\ref{lapDPP}), we then obtain (\ref{ddvsqevol}). The inequality (\ref%
{dDVsqdt}) follows immediately using (\ref{DLapcomm}).
\end{proof}

\begin{corollary}
If the background $G_{2}$-structure $\varphi $ is torsion-free, then $%
\left\vert \nabla V\right\vert ^{2}$ satisfies the following evolution
equation%
\begin{equation}
\frac{\partial \left( \left\vert \nabla V\right\vert ^{2}\right) }{\partial t%
}=\Delta \left\vert \nabla V\right\vert ^{2}-2\left\vert \nabla
^{2}V\right\vert ^{2}+2\left\vert \nabla V\right\vert ^{4}+4\func{Riem}%
\left( \nabla v,\nabla v\right)  \label{ddvstfree}
\end{equation}%
where $\func{Riem}\left( \nabla v,\nabla v\right) =\func{Riem}_{abcd}\left(
\nabla ^{a}v^{c}\right) \left( \nabla ^{b}v^{d}\right) $ for $v=\func{Im}V$.
\end{corollary}

\begin{proof}
If $T=0$, then $D=\nabla .$ Also then $\pi _{7}\func{Riem}=0$, and hence $%
\func{Ric}=0,$ and similarly $\func{div}\func{Riem}=0$. Then, (\ref%
{ddvstfree}) follows immediately from (\ref{ddvsqevol}).
\end{proof}

The expression (\ref{ddvstfree}) is similar to the evolution of the energy
density of harmonic maps in \cite{EellsSampson}, however we have the
additional $\left\vert \nabla V\right\vert ^{4}$ term that is quadratic in
the dependent variable. As it is well-known in the theory of semilinear
PDEs, such quadratic terms in general lead to blow-ups. We can however get
an estimate on the maximal time for which the energy density is finite.

\begin{theorem}
\label{thmLambdaest}Suppose $V\left( t\right) $ is a solution to (\ref%
{heatflow}) on a finite maximal time interval $[0,t_{\max })$. Then for any $%
t\in \lbrack 0,t_{\max })$, 
\begin{equation}
\Lambda \left( t\right) \leq \frac{2R_{1}}{\left( 1+\frac{2R_{1}}{\Lambda
_{0}+R_{2}}\right) e^{-4R_{1}t}-1}-R_{2},  \label{lamt}
\end{equation}%
where $R_{1}$ is a constant multiple of $\sup_{M}\left\vert \func{Riem}%
\right\vert $ and $R_{2}$ is a linear combination of $\sup_{M}\left\vert 
\func{div}\func{Riem}\right\vert $ and $\sup_{M}\left\vert T\right\vert
\left\vert \func{Riem}\right\vert $.
\end{theorem}

\begin{proof}
From (\ref{dDVsqdt}), using Young's inequality, we can say that for any $%
\varepsilon >0$ 
\begin{eqnarray}
\frac{\partial \left\vert DV\right\vert ^{2}}{\partial t} &\leq &\Delta
\left\vert DV\right\vert ^{2}-2\left\vert D^{2}V\right\vert ^{2}+2\left(
\left\vert DV\right\vert ^{4}+2\left( R_{1}+\varepsilon R_{2}\right)
\left\vert DV\right\vert ^{2}+\frac{1}{8\varepsilon }R_{2}\right)   \notag \\
&\leq &\Delta \left\vert DV\right\vert ^{2}-2\left\vert D^{2}V\right\vert
^{2}+2\left( \varepsilon R_{2}+\left\vert DV\right\vert ^{2}\right)
^{2}+4R_{1}\left( \varepsilon R_{2}+\left\vert DV\right\vert ^{2}\right)  
\notag \\
&&-2\varepsilon ^{2}R_{2}^{2}-4\varepsilon R_{1}R_{2}+\frac{1}{4\varepsilon }%
R_{2}.
\end{eqnarray}%
Taking $\varepsilon \ \ $such that $4\varepsilon R_{1}\geq \frac{1}{%
4\varepsilon }$, then redefining $\ R_{2}\ $as $\varepsilon R_{2}$, and
using $h\left( x,t\right) =R_{2}+\left\vert DV\left( x,t\right) \right\vert
^{2}$, we get 
\begin{equation}
\frac{\partial h}{\partial t}\leq \Delta h-2\left\vert D^{2}V\right\vert
^{2}+2h^{2}+4R_{1}h.  \label{ddu0}
\end{equation}%
Note that in the torsion-free case, from (\ref{ddvstfree}), we can set $%
R_{1}=\sup_{M}\left\vert \func{Riem}\right\vert $. Now, $h$ is a subsolution
of the equation 
\begin{equation}
\frac{\partial u}{\partial t}=\Delta u+2u^{2}+4R_{1}u.  \label{ddu}
\end{equation}%
By the Maximum Principle, $h\left( x,t\right) $ is dominated by solutions of
(\ref{ddu}) if $\Lambda \left( x,0\right) \leq u\left( x,0\right) $ for all $%
x.$ Since for $t=0$, $h\left( x,0\right) \leq h\left( 0\right) $, we can
take $u=u\left( t\right) $ with $u\left( 0\right) =h\left( 0\right) =\Lambda
_{0}+R_{2}$. Solving the ODE $\frac{du}{dt}=2u^{2}+4R_{1}u$ with these
initial conditions, then gives us the bound (\ref{lamt}).
\end{proof}

Given a solution $V$ to (\ref{heatflow}), define $\Lambda ^{\left( m\right)
}\left( t\right) =\sup_{x\in M}\left( \left\vert D^{m}V\left( x,t\right)
\right\vert ^{2}\right) .$ Then we have the following estimates.

\begin{theorem}
\label{thmDkVest}For any positive integer $m\geq 2$ there exists a constant $%
C_{m}>0$ that only depends on $M$ and the background $G_{2}$-structure, such
that if $V\left( t\right) $ is a solution to (\ref{heatflow}) for $t\in
\lbrack 0,t_{0})$ with $\Lambda \left( t\right) \leq K$, with $K\geq 1$,
then 
\begin{equation}
\Lambda ^{\left( m\right) }\left( t\right) \leq C_{m}K^{m}\ \text{for }t\in
\lbrack 0,t_{0}).  \label{Dkestimate}
\end{equation}
\end{theorem}

\begin{proof}
Consider first the evolution of $D^{2}V$. From (\ref{dDVdt}) and (\ref%
{DLapcomm}), we can write schematically 
\begin{eqnarray*}
\frac{\partial \left( D^{2}V\right) }{\partial t} &=&D\frac{\partial DV}{%
\partial t} \\
&=&D\Delta _{D}\left( DV\right) +D^{2}\left( \left\vert DV\right\vert
^{2}V\right)  \\
&&+D\left( \func{Riem}\ast DV\right) +D\left( \left( \func{div}\func{Riem}+%
\func{Riem}\ast T\right) \ast V\right) .
\end{eqnarray*}%
Applying (\ref{DLapcomm}) again, we have%
\begin{eqnarray}
\frac{\partial \left( D^{2}V\right) }{\partial t} &=&\Delta _{D}\left(
D^{2}V\right) +\nabla ^{2}\left( \left\vert DV\right\vert ^{2}\right)
V+2\nabla \left( \left\vert DV\right\vert ^{2}\right) DV+\left\vert
DV\right\vert ^{2}D^{2}V  \notag \\
&&+\func{Riem}\ast D^{2}V+\left( \nabla \func{Riem}+\func{Riem}\ast T\right)
\ast DV  \label{dD2V} \\
&&+\left( \nabla \left( \func{div}\func{Riem}\right) +\nabla \func{Riem}\ast
T+\func{Riem}\ast T\ast T+\func{Riem}\ast \nabla T\right) \ast V  \notag
\end{eqnarray}%
and thus, 
\begin{eqnarray}
\frac{\partial \left( \left\vert D^{2}V\right\vert ^{2}\right) }{\partial t}
&=&\Delta \left\vert D^{2}V\right\vert ^{2}-2\left\vert D^{3}V\right\vert
^{2}+2\left\langle \nabla ^{2}\left( \left\vert DV\right\vert ^{2}\right)
V,D^{2}V\right\rangle   \label{dD2Vsq} \\
&&+4\left\langle \nabla \left( \left\vert DV\right\vert ^{2}\right)
DV,D^{2}V\right\rangle +2\left\vert DV\right\vert ^{2}\left\vert
D^{2}V\right\vert ^{2}  \notag \\
&&+\left\langle \func{Riem}\ast D^{2}V,D^{2}V\right\rangle +\left\langle
\left( \nabla \func{Riem}+\func{Riem}\ast T\right) \ast
DV,D^{2}V\right\rangle   \notag \\
&&+\left\langle \left( \nabla \left( \func{div}\func{Riem}\right) +\nabla 
\func{Riem}\ast T+\func{Riem}\ast T\ast T+\func{Riem}\ast \nabla T\right)
\ast V,D^{2}V\right\rangle .  \notag
\end{eqnarray}%
Also, note that 
\begin{subequations}%
\label{derivIds} 
\begin{eqnarray}
\left\langle V,D_{a}D_{b}V\right\rangle  &=&\nabla _{a}\left\langle
V,D_{b}V\right\rangle -\left\langle D_{a}V,D_{b}V\right\rangle
=-\left\langle D_{a}V,D_{b}V\right\rangle  \\
\nabla _{a}\nabla _{b}\left\vert DV\right\vert ^{2} &=&\nabla _{a}\left(
\left\langle D_{b}D_{c}V,D^{c}V\right\rangle +\left\langle
D_{c}V,D_{b}D^{c}V\right\rangle \right)  \\
&=&2\left\langle D_{a}D_{b}D_{c}V,D^{c}V\right\rangle +2\left\langle
D_{b}D_{c}V,D_{a}D^{c}V\right\rangle   \notag \\
\nabla \left( \left\vert DV\right\vert ^{2}\right)  &=&2\left\langle
DV,D^{2}V\right\rangle .
\end{eqnarray}%
\end{subequations}%
Thus, 
\begin{eqnarray*}
\left\langle \nabla ^{2}\left( \left\vert DV\right\vert ^{2}\right)
V,D^{2}V\right\rangle  &=&2\left\langle D_{a}D_{b}D_{c}V,D^{c}V\right\rangle
\left\langle V,D^{a}D^{b}V\right\rangle  \\
&&-2\left\langle D_{b}D_{c}V,D_{a}D^{c}V\right\rangle \left\langle
D^{a}V,D^{b}V\right\rangle 
\end{eqnarray*}%
and hence, 
\begin{eqnarray*}
\left\vert \left\langle \nabla ^{2}\left( \left\vert DV\right\vert
^{2}\right) V,D^{2}V\right\rangle \right\vert  &\leq &2\left\vert
D^{3}V\right\vert \left\vert DV\right\vert ^{3}+2\left\vert
D^{2}V\right\vert ^{2}\left\vert DV\right\vert ^{2} \\
\left\langle \nabla \left( \left\vert DV\right\vert ^{2}\right)
DV,D^{2}V\right\rangle  &\leq &2\left\vert DV\right\vert ^{2}\left\vert
D^{2}V\right\vert ^{2}.
\end{eqnarray*}%
Overall, we then get 
\begin{eqnarray}
\frac{\partial \left( \left\vert D^{2}V\right\vert ^{2}\right) }{\partial t}
&\leq &\Delta \left\vert D^{2}V\right\vert ^{2}-2\left\vert
D^{3}V\right\vert ^{2}+4\left\vert D^{3}V\right\vert \left\vert
DV\right\vert ^{3}+C_{1}\left\vert DV\right\vert ^{2}\left\vert
D^{2}V\right\vert ^{2}  \notag \\
&&+C_{2}\left\vert D^{2}V\right\vert ^{2}+C_{3}\left\vert D^{2}V\right\vert
\left\vert DV\right\vert +C_{4}\left\vert D^{2}V\right\vert .
\label{dD2Vsqineq}
\end{eqnarray}%
Now, using Young's inequality for any $\varepsilon _{1}>1$ we have 
\begin{equation*}
\left\vert D^{3}V\right\vert \left\vert DV\right\vert ^{3}\leq \frac{%
\varepsilon _{1}}{2}\left\vert D^{3}V\right\vert ^{2}+\frac{1}{2\varepsilon
_{1}}\left\vert DV\right\vert ^{6}
\end{equation*}%
and hence (\ref{dD2Vsqineq}) becomes 
\begin{eqnarray*}
\frac{\partial \left( \left\vert D^{2}V\right\vert ^{2}\right) }{\partial t}
&\leq &\Delta \left\vert D^{2}V\right\vert ^{2}-2\left( 1-\varepsilon
_{1}\right) \left\vert D^{3}V\right\vert ^{2}+C_{1}\left( \Lambda \left(
x,t\right) +1\right) \left\vert D^{2}V\right\vert ^{2} \\
&&+C_{2}\left( \Lambda \left( x,t\right) +\Lambda \left( x,t\right)
^{3}\right) .
\end{eqnarray*}%
Now, by hypothesis, $\Lambda \left( x,t\right) \leq K,$ and $K\geq 1$, so we
have 
\begin{equation}
\frac{\partial \left( \left\vert D^{2}V\right\vert ^{2}\right) }{\partial t}%
\leq \Delta \left\vert D^{2}V\right\vert ^{2}-2\left( 1-\varepsilon
_{1}\right) \left\vert D^{3}V\right\vert ^{2}+C_{1}K\left\vert
D^{2}V\right\vert ^{2}+C_{2}K^{3},  \label{D2Vbound}
\end{equation}%
where we assume $\varepsilon _{1}<1$. From (\ref{dDVsqdt}), we also have 
\begin{equation}
\frac{\partial \Lambda \left( x,t\right) }{\partial t}\leq \Delta \Lambda
-2\left\vert D^{2}V\right\vert ^{2}+C_{3}K^{2}.
\end{equation}%
Now let 
\begin{equation}
h=\left( 8K+\Lambda \left( x,t\right) \right) \left\vert D^{2}V\right\vert
^{2}.  \label{hD2Vdef}
\end{equation}%
Then, 
\begin{eqnarray}
\frac{\partial h}{\partial t} &=&\frac{\partial \Lambda \left( x,t\right) }{%
\partial t}\left\vert D^{2}V\right\vert ^{2}+\left( 8K+\Lambda \left(
x,t\right) \right) \frac{\partial \left( \left\vert D^{2}V\right\vert
^{2}\right) }{\partial t} \\
&\leq &\left\vert D^{2}V\right\vert ^{2}\Delta \Lambda \left( x,t\right)
+\left( 8K+\Lambda \left( x,t\right) \right) \Delta \left\vert
D^{2}V\right\vert ^{2}  \notag \\
&&-2\left\vert D^{2}V\right\vert ^{4}-16\left( 1-\varepsilon _{1}\right)
K\left\vert D^{3}V\right\vert ^{2}  \notag \\
&&+C_{1}K^{2}\left\vert D^{2}V\right\vert ^{2}+C_{2}K^{4}  \notag
\end{eqnarray}%
for some new constants $C_{1}$ and $C_{2}$. On the other hand, 
\begin{equation*}
\Delta h=\left\vert D^{2}V\right\vert ^{2}\Delta \Lambda \left( x,t\right)
+\left( 8K+\Lambda \left( x,t\right) \right) \Delta \left\vert
D^{2}V\right\vert ^{2}+2\nabla _{a}\Lambda \nabla ^{a}\left\vert
D^{2}V\right\vert ^{2}
\end{equation*}%
and 
\begin{eqnarray*}
2\nabla _{a}\Lambda \nabla ^{a}\left\vert D^{2}V\right\vert ^{2} &\geq
&-2\left\vert \nabla \left\vert DV\right\vert ^{2}\right\vert \left\vert
\nabla \left\vert D^{2}V\right\vert ^{2}\right\vert  \\
&\geq &-8\left\vert DV\right\vert \left\vert D^{3}V\right\vert \left\vert
D^{2}V\right\vert ^{2} \\
&\geq &-16\varepsilon _{2}K\left\vert D^{3}V\right\vert ^{2}-\frac{1}{%
\varepsilon _{2}}\left\vert D^{2}V\right\vert ^{4},
\end{eqnarray*}%
where we have used Young's Inequality with $\varepsilon _{2}>0$. Thus,
overall, 
\begin{eqnarray}
\left\vert D^{2}V\right\vert ^{2}\Delta \Lambda \left( x,t\right) +\left(
8K+\Lambda \left( x,t\right) \right) \Delta \left\vert D^{2}V\right\vert
^{2} &\leq &\Delta h+16\varepsilon _{2}K\left\vert D^{3}V\right\vert ^{2} \\
&&+\frac{1}{\varepsilon _{2}}\left\vert D^{2}V\right\vert ^{4},  \notag
\end{eqnarray}%
and so we obtain%
\begin{eqnarray}
\frac{\partial h}{\partial t} &\leq &\Delta h+16K\left( \varepsilon
_{1}+\varepsilon _{2}-1\right) K\left\vert D^{3}V\right\vert ^{2}+\left( 
\frac{1}{\varepsilon _{2}}-2\right) \left\vert D^{2}V\right\vert ^{4}  \notag
\\
&&+C_{1}K^{2}\left\vert D^{2}V\right\vert ^{2}+C_{2}K^{4}  \notag \\
&\leq &\Delta h+16K\left( \varepsilon _{1}+\varepsilon _{2}-1\right)
K\left\vert D^{3}V\right\vert ^{2}+\left( \frac{1}{\varepsilon _{2}}%
+\varepsilon _{3}-2\right) \left\vert D^{2}V\right\vert ^{4}  \notag \\
&&+\left( \frac{C_{1}}{4\varepsilon _{2}}+C_{2}\right) K^{4},
\end{eqnarray}%
where we again applied Young's Inequality with some $\varepsilon _{3}>0$.
Then, since $\left\vert D^{2}V\right\vert ^{2}\leq \frac{h}{8K},$ after an
appropriate choice of $\varepsilon _{1},\varepsilon _{2},\varepsilon _{3}$,
we find that there exists a positive constant $C$ such that 
\begin{equation}
\frac{\partial h}{\partial t}\leq \Delta h-\frac{h^{2}}{CK^{2}}+CK^{4}.
\label{D2Vheq}
\end{equation}%
Now considering solutions of the ODE 
\begin{equation*}
\frac{du}{dt}=-\frac{u^{2}}{CK^{2}}+CK^{4},
\end{equation*}%
we find that 
\begin{equation*}
u\leq CK^{3}.
\end{equation*}%
Therefore, by the Maximum Principle, we also find that 
\begin{equation*}
h\leq CK^{3}\text{,}
\end{equation*}%
and hence 
\begin{equation}
\left\vert D^{2}V\right\vert ^{2}\leq CK^{2}  \label{D2VK2}
\end{equation}%
for some other constant $C>1$. Note that using (\ref{D2VK2}) we can rewrite (%
\ref{D2Vbound}) as 
\begin{equation}
\frac{\partial \left( \left\vert D^{2}V\right\vert ^{2}\right) }{\partial t}%
\leq \Delta \left\vert D^{2}V\right\vert ^{2}-2\varepsilon _{1}\left\vert
D^{3}V\right\vert ^{2}+C_{3}K^{3}  \label{D2Vbound2}
\end{equation}%
for $\varepsilon _{1}<1$.

Now assuming bounds for the first and second derivative we will obtain a
bound for $\left\vert D^{3}V\right\vert ^{2}$. From (\ref{dD2V}) and (\ref%
{DLapcomm}), it is not difficult to see that 
\begin{eqnarray}
\frac{\partial \left( D^{3}V\right) }{\partial t} &=&\Delta _{D}\left(
D^{3}V\right) +D^{3}\left( \left\vert DV\right\vert ^{2}V\right) +\func{Riem}%
\ast D^{3}V  \label{ddkp1V} \\
&&+R_{2}\ast D^{2}V+R_{2}\ast DV+R_{0}  \notag
\end{eqnarray}%
where $R_{i}$ for $i=0,1,...,k-1$ are some tensors that combine derivatives
of $\func{Riem}$ and $T$. Now, taking the inner product of (\ref{ddkp1V})
with $D^{3}V$ and applying (\ref{lapDPP}), we obtain 
\begin{eqnarray}
\frac{\partial \left\vert D^{3}V\right\vert ^{2}}{\partial t} &\leq &\Delta
\left\vert D^{3}V\right\vert ^{2}-2\left\vert D^{4}V\right\vert
^{2}+\left\langle D^{3}\left( \left\vert DV\right\vert ^{2}V\right)
,D^{3}V\right\rangle   \label{ddkvsq} \\
&&+C\left\vert D^{3}V\right\vert ^{2}+C\left\vert D^{2}V\right\vert
\left\vert D^{3}V\right\vert +C\left\vert DV\right\vert \left\vert
D^{3}V\right\vert +C\left\vert D^{3}V\right\vert .  \notag
\end{eqnarray}%
Let us focus on the third term on the right-hand side of (\ref{ddkvsq}).
Schematically, ignoring indices on derivatives, we have 
\begin{eqnarray}
D^{3}\left( \left\vert DV\right\vert ^{2}V\right)  &=&\left( \nabla
^{3}\left\vert DV\right\vert ^{2}\right) V+3\left( \nabla ^{2}\left\vert
DV\right\vert ^{2}\right) DV+3\left( \nabla \left\vert DV\right\vert
^{2}\right) D^{2}V  \label{D2dvsqv} \\
&&+\left\vert DV\right\vert ^{2}D^{3}V  \notag
\end{eqnarray}%
and since $D^{3}\left\langle V,V\right\rangle =0$ we also have 
\begin{equation*}
\left\langle V,D^{3}V\right\rangle =-3\left\langle DV,D^{2}V\right\rangle .
\end{equation*}%
Thus, 
\begin{eqnarray}
\left\langle D^{3}\left( \left\vert DV\right\vert ^{2}V\right)
,D^{3}V\right\rangle  &=&-3\left( \nabla ^{3}\left\vert DV\right\vert
^{2}\right) \left\langle DV,D^{2}V\right\rangle   \label{D2dvsqv2} \\
&&+3\left( \nabla ^{2}\left\vert DV\right\vert ^{2}\right) \left\langle
DV,D^{3}V\right\rangle   \notag \\
&&+3\left( \nabla \left\vert DV\right\vert ^{2}\right) \left\langle
D^{2}V,D^{3}V\right\rangle +\left\vert DV\right\vert ^{2}\left\vert
D^{3}V\right\vert ^{2},  \notag
\end{eqnarray}%
and applying bounds $\left\vert DV\right\vert \leq CK^{\frac{1}{2}}$, $%
\left\vert D^{2}V\right\vert \leq CK$, we, get 
\begin{eqnarray*}
\left\vert \left\langle D^{3}\left( \left\vert DV\right\vert ^{2}V\right)
,D^{3}V\right\rangle \right\vert  &\leq &CK^{2}\left\vert D^{4}V\right\vert
+CK\left\vert D^{3}V\right\vert ^{2}+CK^{4} \\
&\leq &2\varepsilon \left\vert D^{4}V\right\vert ^{2}+CK\left\vert
D^{3}V\right\vert ^{2}+CK^{4},
\end{eqnarray*}%
where we also used Young's Inequality on the first term. So overall, we have 
\begin{equation}
\frac{\partial \left\vert D^{3}V\right\vert ^{2}}{\partial t}\leq \Delta
\left\vert D^{3}V\right\vert ^{2}-2\left( 1-\varepsilon \right) \left\vert
D^{4}V\right\vert ^{2}+C_{1}K\left\vert D^{3}V\right\vert ^{2}+C_{2}K^{4}.
\label{D3Vbound}
\end{equation}%
Similarly as before, let 
\begin{equation}
h=\left( 8L+\left\vert D^{2}V\right\vert ^{2}\right) \left\vert
D^{3}V\right\vert ^{2},  \label{hD3def}
\end{equation}%
where $L$ is a constant such that $K^{2}\leq L\leq CK^{2}$ (which is
possible since $C>1$) and $\left\vert D^{2}V\right\vert ^{2}\leq L.$ Then,
using (\ref{D3Vbound})\ and (\ref{D2Vbound2}), we find that 
\begin{equation}
\frac{\partial h}{\partial t}\leq \Delta h-\frac{h^{2}}{CK^{4}}+CK^{6},
\label{d3vheq}
\end{equation}%
and this then gives us the bound 
\begin{equation}
\left\vert D^{3}V\right\vert ^{2}\leq CK^{3}
\end{equation}%
for some constant $C$. By induction can similarly obtain bounds (\ref%
{Dkestimate}) for higher derivatives.
\end{proof}

\begin{proof}[Proof of Theorem \protect\ref{thmMain}]
Suppose a solution $V\left( t\right) $ to (\ref{heatflow}) exists on the
finite maximal time interval $[0,t_{\max }).$ We will proceed by
contradiction to prove (\ref{limLambdainf}). Suppose (\ref{limLambdainf})
does not hold. This implies that there exists a constant $K$ such that 
\begin{equation}
\sup_{M\times \lbrack 0,t_{\max })}\Lambda \left( x,t\right) \leq K.
\label{contraassum}
\end{equation}%
We then know from (\ref{Dkestimate}) that for some constant $C_{2}>0$ 
\begin{equation*}
\sup_{M\times \lbrack 0,t_{\max })}\Lambda ^{\left( 2\right) }\left(
x,t\right) \leq C_{2}K^{2}.
\end{equation*}%
So in particular, $\left\vert D^{2}V\right\vert $ is bounded, and thus from
the flow equation (\ref{heatflow}), there exists a constant $C>0$ such that 
\begin{equation*}
\sup_{M\times \lbrack 0,t_{\max })}\left\vert \frac{\partial V}{\partial t}%
\right\vert \leq C.
\end{equation*}%
Then, for any $0<t_{1}<t_{2}<t_{\max },$ we have 
\begin{equation}
\left\vert V\left( t_{2}\right) -V\left( t_{1}\right) \right\vert \leq
\int_{t_{1}}^{t_{2}}\left\vert \frac{\partial V}{\partial t}\right\vert
dt\leq C\left( t_{2}-t_{1}\right) .
\end{equation}%
Therefore, we see that as $t\longrightarrow t_{\max }$, the octonion
sections $V\left( t\right) $ converge continuously to a section $V\left(
t_{\max }\right) .$ Clearly, this will also have unit norm. Locally, for
some $0<t<t_{\max }$ we can then write%
\begin{equation}
V\left( t_{\max }\right) =V\left( t\right) +\int_{t}^{t_{\max }}\left(
\Delta _{D}V\left( s\right) +\left\vert DV\left( s\right) \right\vert
^{2}V\left( s\right) \right) ds.
\end{equation}%
Now, by Theorem \ref{thmDkVest}, all derivatives of $V$ are uniformly
bounded, hence all derivatives of $V\left( t_{\max }\right) $ are also
bounded. Thus, $V\left( t_{\max }\right) $ is a smooth section and $V\left(
t\right) $ converges to it uniformly in any $C^{m}$-norm as $%
t\longrightarrow t_{\max }$. Now we have smoothly extended the flow from $%
[0,t_{\max })$ to $t_{\max }$. However, using short-time existence and
uniqueness of the flow, we can uniquely extend it further starting from $%
t=t_{\max }$ to $t=t_{\max }+\varepsilon $ for some $\varepsilon >0$.
Therefore, the flow exists on $[0,t_{\max }+\varepsilon )$ and this
contradicts the maximality of $t_{\max }$. We then find that (\ref%
{contraassum}) fails, and can conclude that 
\begin{equation}
\lim_{t\longrightarrow t_{\max }}\Lambda \left( t\right) =\infty .
\end{equation}

From (\ref{ddu0}), we see that 
\begin{equation*}
\frac{d\left( \Lambda \left( t\right) +C_{0}\right) }{dt}\leq 2\left(
\Lambda \left( t\right) +C_{0}\right) ^{2}
\end{equation*}%
where $C_{0}=R_{1}+R_{2}$ is a constant that depends on the curvature and
the torsion of the background $G_{2}$-structure. Thus, 
\begin{equation*}
\frac{d\left( \Lambda \left( t\right) +C_{0}\right) ^{-1}}{dt}\geq -2
\end{equation*}%
and thus, integrating, and taking the limit as $t\longrightarrow t_{\max }$,
we obtain 
\begin{equation}
\Lambda \left( t\right) \geq \frac{1}{2\left( t_{\max }-t\right) }-C_{0}
\end{equation}%
for all $t\in \lbrack 0,t_{\max }).$ Thus, we obtain (\ref{lamdbatineq}).
\end{proof}

\section{Monotonicity}

\label{secMonotonicty}\setcounter{equation}{0}In order to be able to get a
better control on the flow, it is useful to find quantities that are
monotonic, or otherwise well-behaved along the flow. Following Hamilton \cite%
{HamiltonMonotonicity}, let $k$ be a positive scalar solution of the
backwards heat equation 
\begin{equation}
\frac{\partial k}{\partial t}=-\Delta k  \label{backwardheat}
\end{equation}%
for $0\leq t\leq t_{0}$, with some initial condition at $t=t_{0}$ and
evolving towards $t=0$, such that $\int_{M}k\func{vol}=1$. Then, consider
the quantity%
\begin{equation}
Z\left( t\right) =\left( t_{0}-t\right) \int_{M}\left\vert DV\right\vert
^{2}k\func{vol}.  \label{Zdef}
\end{equation}

\begin{theorem}
\label{thmMonotonicity} Suppose $V$ is a solution of the flow (\ref{heatflow}%
) for $0\leq t<t_{0}$ with initial energy $\mathcal{E}\left( 0\right) =%
\mathcal{E}_{0}$. Then, there exists a constant $C>0$, that only depends on
the background geometry, such that for any $t\ $and $\tau $ satisfying $%
t_{0}-1\leq \tau \leq t<t_{0}$, $Z\left( t\right) $ satisfies the following
relation%
\begin{equation}
Z\left( t\right) \leq CZ\left( \tau \right) +C\left( t-\tau \right) \left( 
\mathcal{E}_{0}+\mathcal{E}_{0}^{\frac{1}{2}}\right)   \label{Ztmonotonicity}
\end{equation}
\end{theorem}

\begin{proof}
Differentiating $Z(t)$, we find 
\begin{eqnarray}
\frac{dZ}{dt} &=&-\int_{M}\left\vert DV\right\vert ^{2}k\func{vol}+\left(
t_{0}-t\right) \int_{M}\frac{\partial }{\partial t}\left( \left\vert
DV\right\vert ^{2}\right) k\func{vol}  \label{DZdt} \\
&&-\left( t_{0}-t\right) \int_{M}\left\vert DV\right\vert ^{2}\Delta k\func{%
vol}.  \notag
\end{eqnarray}%
Consider the second term on the right-hand side of (\ref{DZdt}). We use the
evolution equation (\ref{dDVsqnocurv}) for $\left\vert DV\right\vert ^{2}$
and then integrate by parts:%
\begin{eqnarray}
\int_{M}\frac{\partial }{\partial t}\left( \left\vert DV\right\vert
^{2}\right) k\func{vol} &=&2\int_{M}\left( \left\langle D_{i}\left( \Delta
_{D}V\right) ,D^{i}V\right\rangle +\left\vert DV\right\vert ^{4}\right) k%
\func{vol}  \notag \\
&=&-2\int_{M}\left( \left\vert \Delta _{D}V\right\vert ^{2}-\left\vert
DV\right\vert ^{4}\right) k\func{vol}  \label{DZdt2a} \\
&&-2\int_{M}\left\langle \Delta _{D}V,D^{i}V\right\rangle \nabla _{i}k\func{%
vol}.  \notag
\end{eqnarray}%
Let us now rewrite (\ref{DZdt2a}) by completing the square $\left\vert
\Delta _{D}V+\left\vert DV\right\vert ^{2}V+\frac{1}{k}\nabla
_{i}kD^{i}V\right\vert ^{2}$: 
\begin{eqnarray}
\int_{M}\frac{\partial }{\partial t}\left( \left\vert DV\right\vert
^{2}\right) k\func{vol} &=&-2\int \left\vert \Delta _{D}V+\left\vert
DV\right\vert ^{2}V+\frac{1}{k}\nabla _{i}kD^{i}V\right\vert ^{2}k\func{vol}
\\
&&+2\int_{M}\left\langle \Delta _{D}V,D^{i}V\right\rangle \nabla _{i}k\func{%
vol}+2\int_{M}\frac{1}{k}\left( \nabla _{i}k\nabla _{j}k\right) \left\langle
D^{i}V,D^{j}V\right\rangle \func{vol}  \notag
\end{eqnarray}%
and finally, let us integrate the second term by parts, so that overall, we
get 
\begin{eqnarray}
\int_{M}\frac{\partial }{\partial t}\left( \left\vert DV\right\vert
^{2}\right) k\func{vol} &=&-2\int_{M}\left\vert \Delta _{D}V+\left\vert
DV\right\vert ^{2}V+\frac{1}{k}\nabla _{i}kD^{i}V\right\vert ^{2}k\func{vol}
\\
&&-2\int_{M}\left( \left\langle D^{j}V,D_{j}D^{i}V\right\rangle \nabla
_{i}k+\left\langle D^{j}V,D^{i}V\right\rangle \nabla _{j}\nabla _{i}k\right) 
\func{vol}  \notag \\
&&+2\int_{M}\frac{1}{k}\left( \nabla _{i}k\nabla _{j}k\right) \left\langle
D^{i}V,D^{j}V\right\rangle \func{vol}.  \notag
\end{eqnarray}%
Now consider the third term of the right-hand side of (\ref{DZdt}).
Integrating by parts, we get 
\begin{eqnarray}
\int_{M}\left\vert DV\right\vert ^{2}\Delta k\func{vol} &=&-2\int_{M}\left%
\langle D_{j}V,D^{i}D^{j}V\right\rangle \nabla _{i}k\func{vol}  \notag \\
&=&-2\int_{M}\left\langle D_{j}V,D^{j}D^{i}V+F^{ij}\left( V\right)
\right\rangle \nabla _{i}k\func{vol}  \notag \\
&=&-2\int_{M}\left\langle D_{j}V,D^{j}D^{i}V\right\rangle \nabla _{i}k\func{%
vol}  \label{DZdt2c} \\
&&+2\int_{M}\left\langle D^{i}V,F_{ij}\left( V\right) \right\rangle \nabla
^{j}k\func{vol}.  \notag
\end{eqnarray}%
Combining all the terms, (\ref{DZdt})\ becomes%
\begin{eqnarray}
\frac{dZ}{dt}+2W &=&-2\left( t_{0}-t\right) \int_{M}\left[ \nabla _{i}\nabla
_{j}k-\frac{1}{k}\left( \nabla _{i}k\nabla _{j}k\right) +\frac{kg_{ij}}{%
2\left( t_{0}-t\right) }\right] \left\langle D^{i}V,D^{j}V\right\rangle 
\func{vol}  \notag \\
&&-2\left( t_{0}-t\right) \int_{M}\left\langle D^{i}V,F_{ij}\left( V\right)
\right\rangle \nabla ^{j}k\func{vol},  \label{DZdt3}
\end{eqnarray}%
where we set 
\begin{equation}
W=\left( t_{0}-t\right) \int_{M}\left\vert \Delta _{D}V+\left\vert
DV\right\vert ^{2}V+\frac{1}{k}\nabla _{i}kD^{i}V\right\vert ^{2}k\func{vol}.
\label{Wdef}
\end{equation}%
Applying integration by parts to the last term in (\ref{DZdt3}), we get%
\begin{eqnarray}
\frac{dZ}{dt}+2W &=&-2\left( t_{0}-t\right) \int_{M}\left[ \nabla _{i}\nabla
_{j}k-\frac{1}{k}\left( \nabla _{i}k\nabla _{j}k\right) +\frac{kg_{ij}}{%
2\left( t_{0}-t\right) }\right] \left\langle D^{i}V,D^{j}V\right\rangle 
\func{vol}  \notag \\
&&-2\left( t_{0}-t\right) \int_{M}\left\langle D^{j}V,D^{i}\left(
F_{ij}\left( V\right) \right) \right\rangle k\func{vol}  \label{DZdt4} \\
&&-\left( t_{0}-t\right) \int_{M}\left\vert F\left( V\right) \right\vert
^{2}k\func{vol}.  \notag
\end{eqnarray}%
For the first term on the right-hand side of (\ref{DZdt4}) we apply
Hamilton's matrix Harnack inequality from \cite{HamiltonMatrixHarnack}:
there exist constants $B$ and $C$ that depend only on the geometry of $M$
such that 
\begin{equation}
\nabla _{i}\nabla _{j}k-\frac{1}{k}\left( \nabla _{i}k\nabla _{j}k\right) +%
\frac{kg_{ij}}{2\left( t_{0}-t\right) }\geq -C\left( 1+k\ln \left( \frac{B}{%
\left( t_{0}-t\right) ^{\frac{7}{2}}}\right) \right) g_{ij}.
\label{MatrixHarnack}
\end{equation}%
Note that the trace of this estimate gives the well-known Harnack estimate
by Li and Yau \cite{LiYau1}.

Let us now consider the second term in (\ref{DZdt4}). From (\ref{DivF}), we
know that%
\begin{eqnarray}
D^{a}\left( F_{ab}\left( V\right) \right)  &=&\left( \nabla ^{a}\func{Riem}%
_{ab}\right) \left( V\right) +\left( \func{Riem}_{ab}\nabla ^{a}\right)
\left( V\right) -\func{Riem}_{ab}\left( V\right) T^{a}  \notag \\
&&-\frac{1}{4}\left( \nabla ^{a}V\right) \left( \pi _{7}\func{Riem}\right)
_{ab}-\frac{1}{4}V\left( D^{a}\left( \pi _{7}\func{Riem}\right) _{ab}\right) 
\label{DivFV}
\end{eqnarray}%
and therefore, 
\begin{equation}
\left\vert \left\langle D^{j}V,D^{i}\left( F_{ij}\left( V\right) \right)
\right\rangle \right\vert \leq R_{1}\left( \left\vert DV\right\vert
^{2}+\left\vert DV\right\vert \right) ,  \label{DZdt4c}
\end{equation}%
where $R_{1}$ is a constant that depends on the curvature and torsion. Now
using (\ref{MatrixHarnack}) and (\ref{DZdt4c}) in (\ref{DZdt4}), and noting
that $\int_{M}k\func{vol}=1$, we have 
\begin{equation}
\frac{dZ}{dt}+2W\leq 2\left( t_{0}-t\right) R\left( \mathcal{E}\left(
t\right) +1\right) +2R\left( 1+\ln \left( \frac{B}{\left( t_{0}-t\right) ^{%
\frac{7}{2}}}\right) \right) Z,  \label{DZdt5}
\end{equation}%
where for convenience we now take $R$ to be the greater of $R_{1}$ and $C$.
Let 
\begin{equation}
q=\left( t_{0}-t\right) \left( \frac{9}{2}+\ln \left( \frac{B}{\left(
t_{0}-t\right) ^{\frac{7}{2}}}\right) \right) ,  \label{qdef}
\end{equation}%
so that 
\begin{equation}
\frac{dq}{dt}=-\left( 1+\ln \left( \frac{B}{\left( t_{0}-t\right) ^{\frac{7}{%
2}}}\right) \right)   \label{dqdt}
\end{equation}%
and hence, 
\begin{equation}
\frac{d}{dt}\left( e^{2Rq}Z\right) +2e^{2Rq}W\leq 2\left( t_{0}-t\right)
Ce^{2Rq}\left( \mathcal{E}\left( t\right) +\mathcal{E}\left( t\right) ^{%
\frac{1}{2}}\right) \leq 2Ce^{2Rq}\left( \mathcal{E}_{0}+\mathcal{E}_{0}^{%
\frac{1}{2}}\right)   \label{DZdt6}
\end{equation}%
for $t_{0}-1\leq t<t_{0}$. We can always take $B$ to be large enough so that 
$q\geq 0$ and we can also bound $e^{2Rq}$. Now integrating from $\tau $ to $%
t,$ we find that for any $t_{0}-1\leq \tau \leq t<t_{0}$ 
\begin{eqnarray}
Z\left( t\right)  &\leq &e^{2R\left( q\left( \tau \right) -q\left( t\right)
\right) }Z\left( \tau \right) +2Ce^{-2Rq\left( t\right) }\left( t-\tau
\right) \left( \mathcal{E}_{0}+\mathcal{E}_{0}^{\frac{1}{2}}\right)   \notag
\\
&\leq &CZ\left( \tau \right) +C\left( t-\tau \right) \left( \mathcal{E}_{0}+%
\mathcal{E}_{0}^{\frac{1}{2}}\right) ,  \label{Zineq}
\end{eqnarray}%
thus completing the proof.
\end{proof}

\begin{remark}
In \cite{HamiltonMatrixHarnack}, it is shown that in the case when $\nabla 
\func{Ric}=0$ and the sectional curvature of $M$ is non-negative, the
quantity of the left-hand side of (\ref{MatrixHarnack}) is actually
non-negative, and in \cite{HamiltonMonotonicity}, this leads to the
corresponding quantity $Z$ for the harmonic map flow and the Yang-Mills flow
to be monotonically decreasing along the flow. In our case, we have an
additional curvature term in (\ref{DZdt4}), which doesn't immediately give a
non-positive term in this case. On the contrary, in this case, it gives a
non-negative $\func{Riem}\left( \nabla v,\nabla v\right) $ term. Therefore,
it is not clear if there are some reasonable conditions under which $Z\left(
t\right) $ is monotonically decreasing.
\end{remark}

\section{$\protect\varepsilon $-regularity}

\label{secEpsReg}\setcounter{equation}{0}In this section we will use the
results on the behavior of $Z\left( t\right) $ from the previous section as
well as the a priori estimates from section \ref{secHeat}, to obtain an $%
\varepsilon $-regularity result and from it, long-time existence for small
initial energy density (i.e. small pointwise torsion).

Let $p_{x_{0},t_{0}}\left( x,t\right) $ be the backward heat kernel on $M$, that is,
the solution of the backward heat equation (\ref{backwardheat}) for $0\leq
t\leq t_{0}$ that converges to a delta function at $\left( x,t\right)
=\left( x_{0},t_{0}\right) $. Then, given a time-dependent octonion section $%
V\left( x,t\right) $, define the $\mathcal{F}$-functional%
\begin{equation}
\mathcal{F}\left( x_{0},t_{0},t\right) =\left( t_{0}-t\right)
\int_{M}\left\vert DV\left( x,t\right) \right\vert ^{2}p_{x_{0},t_{0}}\left(
x,t\right) \func{vol}\left( x\right) .  \label{Ffunctional}
\end{equation}%
Clearly this is just $Z$ with a particular choice of the backward heat
equation solution $k$. The key result in this section is the following.

\begin{theorem}
\label{thmEpsRegular}Given $\mathcal{E}_{0}$, there exist $\varepsilon >0\ $%
and $\beta >0,\ $both depending on $M$ and $\beta $ also depending on $%
\mathcal{E}_{0}$, such that if $V$ is a solution of the flow (\ref{heatflow}%
) on $M\times \lbrack 0,t_{0})\ $with energy bounded by $\mathcal{E}_{0}$,
and if 
\begin{equation}
\mathcal{F}\left( x_{0},t_{0},t\right) \leq \varepsilon  \label{epsreghypo}
\end{equation}%
for $t\in \lbrack t_{0}-\beta ,t_{0}),$ then $V$ extends smoothly to $%
U_{x_{0}}\times \lbrack 0,t_{0}]$ for some neighborhood $U_{x_{0}}$ of $%
x_{0} $ with $\left\vert DV\right\vert $ bounded uniformly.
\end{theorem}

Before we go on to prove Theorem \ref{thmEpsRegular}, here is an important
corollary.

\begin{corollary}
\label{corrGlobalExist}There exists an $\varepsilon >0$ such that if the
initial energy density $\Lambda _{0}=\left\vert DV\right\vert ^{2}\ $%
satisfies $\Lambda _{0}<\varepsilon ,$ then a solution $V$ of the flow (\ref%
{heatflow}) exists for all $t\geq 0$. The limit $V_{\infty
}=\lim_{t\longrightarrow \infty }V\left( t\right) $ corresponds to a $G_{2}$%
-structure with divergence-free torsion.
\end{corollary}

\begin{proof}
Suppose $V$ is a solution of the flow (\ref{heatflow}) on a maximal time
interval $[0,t_{\max })$ with initial energy $\mathcal{E}_{0}$. By Theorem %
\ref{thmMonotonicity}, $\mathcal{F}$ satisfies the following inequality for
any $t\ $and $\tau $ satisfying $t_{\max }-1\leq \tau \leq t<t_{\max }$ and
any $x_{0}\in M$ 
\begin{equation}
\mathcal{F}\left( x_{0},t_{\max },t\right) \leq C\mathcal{F}\left(
x_{0},t_{\max },\tau \right) +C\left( t-\tau \right) \left( \mathcal{E}_{0}+%
\mathcal{E}_{0}^{\frac{1}{2}}\right) .  \label{Fx0tmaxineq}
\end{equation}%
Using standard properties of the heat kernel, for some constant $C$ we have 
\begin{equation*}
\mathcal{F}\left( x_{0},t_{\max },\tau \right) \leq \frac{C}{\left( t_{\max
}-\tau \right) ^{\frac{5}{2}}}\mathcal{E}\left( \tau \right)
\end{equation*}%
If $t_{\max }\geq 1$, then set $\tau =t_{\max }-1$, and then we get a bound
on $\mathcal{F}$ in terms of $\mathcal{E}$. Otherwise, set for example $\tau
=\frac{t_{\max }}{2}$, and from (\ref{lamdbatineq}) we have $\frac{1}{%
t_{\max }}\leq 2\left( \Lambda _{0}+C_{0}\right) $ for a constant $C_{0}$
that only depends on the background geometry, hence in this case, 
\begin{equation*}
\mathcal{F}\left( x_{0},t_{\max },\tau \right) \leq C\left( \Lambda
_{0}+C_{0}\right) ^{\frac{5}{2}}\mathcal{E}\left( \tau \right) .
\end{equation*}%
Now, $\mathcal{E}\left( \tau \right) \leq \mathcal{E}_{0}\leq \Lambda
_{0}Vol\left( M\right) $, where $Vol\left( M\right) $ is the total volume of
the manifold, so overall from (\ref{Fx0tmaxineq}) we obtain a bound on $%
\mathcal{F}\left( x_{0},t_{\max },t\right) $ in terms of $\Lambda _{0}$.
Hence, choosing $\Lambda _{0}$ small enough, the conditions of Theorem \ref%
{thmEpsRegular} are satisfied, and the solution extends smoothly to $\left[
0,t_{\max }\right] $. Restarting the flow from $t=t_{\max }$, with initial
energy $\mathcal{E}\left( t_{\max }\right) \leq \mathcal{E}_{0}$, by
short-time existence we can then extend it to $[0,t_{\max }+\varepsilon )\ \ 
$for some $\varepsilon >0$, thus contradicting the maximality of $t_{\max }$.

Now we have a solution that exists for all $t>0,$ with $\left\vert
DV\right\vert $ and, from Theorem \ref{thmDkVest}, all higher derivatives
bounded uniformly. This means that choosing $\Lambda _{0}$ sufficiently
small, we can make sure that $\left\vert DV\right\vert $ is also
sufficiently small, so that it satisfies the conditions of Corollary \ref%
{corrDivTdecay} for all time. As in \cite{DGKisoflow}, this then implies
that $\int_{M}\left\vert \func{div}T^{\left( V\right) }\right\vert ^{2}\func{%
vol}\longrightarrow 0$ exponentially and hence, $V\left( t\right) $
converges in $L^{1}$ to a unique limit $V_{\infty }.$ By uniform bounds on
the derivatives, the limit is then smooth and has $\func{div}T^{\left(
V\right) }=0$.
\end{proof}

To prove Theorem \ref{thmEpsRegular}, similarly as in \cite{GraysonHamilton}%
, we need to carefully understand the local behavior of solutions to the
flow (\ref{heatflow}).

\begin{definition}
For any $x_{0}\in M$ and $t_{0}\in \mathbb{R}$, define a parabolic cylinder $%
P_{r}\left( x_{0},t_{0}\right) =\bar{B}_{r}\left( x_{0}\right) \times \left[
t_{0}-r^{2},t_{0}\right] ,$ where $\bar{B}_{r}\left( x_{0}\right) $ is a
closed geodesic ball of radius $r$ centered at $x_{0}$.
\end{definition}

We have the following useful Lemma from \cite{GraysonHamilton}.

\begin{lemma}[{\protect\cite[Lemma 2.1]{GraysonHamilton}}]
\label{lemGraysonHamilton}Let $M$ be a compact manifold. There exists a
constant $s>0,$ and for every $\gamma <1,$ a constant $C_{\gamma }$, such
that if $h$ is a smooth function satisfying 
\begin{equation}
\frac{\partial h}{\partial t}\leq \Delta h-h^{2}  \label{dhhsqineq}
\end{equation}%
whenever $h\geq 0\ $in $P_{r}\left( x_{0},t_{0}\right) $ for some $r\leq s$,
then 
\begin{equation}
h\leq C_{\gamma }\left( \frac{1}{r^{2}}+\frac{1}{t}\right)
\end{equation}%
on $P_{\gamma r}\left( x_{0},t_{0}\right) $.
\end{lemma}

Lemma \ref{lemGraysonHamilton} can be used to modify the proof of Theorem %
\ref{thmDkVest} to give a local version on a parabolic cylinder. Define $%
\Lambda _{B_{r}\left( x_{0}\right) }\left( t\right) =\sup_{x\in B_{r}\left(
x_{0}\right) }\Lambda \left( x,t\right) $ and $\Lambda _{B_{r}\left(
x_{0}\right) }^{\left( m\right) }\left( t\right) =\sup_{x\in B_{r}\left(
x_{0}\right) }\left( \left\vert D^{m}V\left( x,t\right) \right\vert
^{2}\right) $.

\begin{theorem}
\label{thmLocDkVest}There exists a constant $s>0$ and, for any positive
integer $m\geq 2$, a constant $C_{m},$ that only depend on $M$ and the
background $G_{2}$-structure, such that, if $V\left( t\right) $ is a
solution to (\ref{heatflow}) in a parabolic cylinder $P_{r}\left(
x_{0},t_{0}\right) $ for $r<\min \left\{ s,1\right\} $ and satisfies $%
\Lambda _{B_{r}\left( x_{0}\right) }\leq K$ for $K>\frac{1}{r^{2}}$, then 
\begin{equation}
\Lambda _{B_{r_{k}}\left( x_{0}\right) }^{\left( m\right) }\left( t\right)
\leq C_{m}K^{m}\ \text{on }P_{r_{k}}\left( x_{0},t_{0}\right) ,
\end{equation}%
where $r_{k}=2^{1-k}r$.
\end{theorem}

\begin{proof}
The proof is essentially the same as that of Theorem \ref{thmDkVest}. As in 
\cite{GraysonHamilton}, the main difference is that when we obtain differential
inequalities (\ref{D2Vheq}) and (\ref{d3vheq}) for $h$ and for $%
\left\vert D^{2}V\right\vert ^{2}$ and $\left\vert D^{3}V\right\vert ^{2}$,
respectively, we need to make further changes of variables to get these
inequalities into the form (\ref{dhhsqineq}). Then, rather than using the
Maximum Principle directly, we need to apply Lemma \ref{lemGraysonHamilton}.
In particular, when proving the bound for $\left\vert D^{2}V\right\vert ^{2}$%
, we take $h=\left( 8K+\Lambda \left( x,t\right) \right) \left\vert
D^{2}V\right\vert ^{2}$ as in (\ref{hD2Vdef}) and then as in (\ref{D2Vheq}),
we obtain 
\begin{equation}
\frac{\partial h}{\partial t}\leq \Delta h-\frac{h^{2}}{CK^{2}}+CK^{4}.
\end{equation}%
Now, let 
\begin{equation}
\tilde{h}=\frac{h}{CK^{2}}-K  \label{htilde}
\end{equation}%
for the same constant $C$. Hence, we have 
\begin{eqnarray*}
\frac{\partial \tilde{h}}{\partial t} &\leq &\frac{1}{CK^{2}}\Delta h-\frac{%
h^{2}}{C^{2}K^{4}}+K^{2} \\
&\leq &\Delta \tilde{h}-\tilde{h}^{2}-2K\tilde{h}
\end{eqnarray*}%
and therefore, when $\tilde{h}\geq 0$ we have 
\begin{equation*}
\frac{\partial \tilde{h}}{\partial t}\leq \Delta \tilde{h}-\tilde{h}^{2}.
\end{equation*}%
Therefore, applying Lemma \ref{lemGraysonHamilton}\ with $\gamma =\frac{1}{2}
$, we find that for some constant $C$, on $P_{\frac{r}{2}}\left(
x_{0},t_{0}\right) $ we have 
\begin{equation}
\tilde{h}\leq \frac{C}{r^{2}}\leq CK
\end{equation}%
and thus 
\begin{equation*}
h\leq CK^{3}\text{,}
\end{equation*}%
and for some constant $C_{2}$, 
\begin{equation*}
\Lambda _{B_{r_{2}}\left( x_{0}\right) }^{\left( 2\right) }\left( t\right)
\leq C_{2}K^{2}.
\end{equation*}%
A similar argument follows for higher derivatives.
\end{proof}

We now need a lemma similar to Lemma 3.1 in \cite{GraysonHamilton}.

\begin{lemma}
\label{lemFdelta}There exist constants $\delta >0$ and $\gamma \in \left(
0,1\right) $ that depend only on $M\ $and the background $G_{2}$-structure$,$
such that, if $V$ is a solution of the flow (\ref{heatflow}) in a parabolic
cylinder $P_{r}\left( x_{0},t_{0}\right) $ for $r\leq 1$ such that 
\begin{equation*}
\left\vert DV\left( x_{0},t_{0}\right) \right\vert =\frac{1}{r}
\end{equation*}%
and 
\begin{equation*}
\left\vert DV\left( x,t\right) \right\vert \leq \frac{2}{r}
\end{equation*}%
for all $\left( x,t\right) \in $ $P_{r}\left( x_{0},t_{0}\right) ,$ then for 
$\theta =t_{0}-\gamma r^{2}$ we have%
\begin{equation}
\mathcal{F}\left( x_{0},t_{0},\theta \right) \geq \delta .  \label{Fdelta}
\end{equation}
\end{lemma}

\begin{proof}
From (\ref{dDVdt}), we have%
\begin{eqnarray}
\left\vert \frac{\partial }{\partial t}\left( DV\right) \right\vert &\leq
&\left\vert \Delta _{D}\left( DV\right) \right\vert +C_{1}\left\vert
DV\right\vert +C_{2}  \label{dtdv} \\
&&+2\left\vert D^{2}V\right\vert \left\vert DV\right\vert +\left\vert
DV\right\vert ^{3}.  \notag
\end{eqnarray}
\end{proof}

By hypothesis, $\left\vert DV\right\vert ^{2}$ is bounded on the parabolic
cylinder $P_{r}\left( x_{0},t_{0}\right) $ by $\frac{4}{r^{2}}$, hence by
Theorem \ref{thmLocDkVest}, there exist constants $C_{2}$ and $C_{3}$ such
that $\left\vert D^{2}V\right\vert ^{2}\leq \frac{C_{2}}{r^{4}}$ on $P_{%
\frac{r}{2}}\left( x_{0},t_{0}\right) $ and $\left\vert D^{3}V\right\vert
^{2}\leq \frac{C_{3}}{r^{6}}$ on $P_{\frac{r}{4}}\left( x_{0},t_{0}\right) $%
. Therefore, from (\ref{dtdv}) we find that on $P_{\frac{r}{4}}\left(
x_{0},t_{0}\right) $, 
\begin{equation}
\left\vert \frac{\partial }{\partial t}\left( DV\right) \right\vert \leq 
\frac{C}{r^{3}}  \label{ddtDV}
\end{equation}%
for some constant $C>0$. Note that the octonionic derivative $D$, being
metric-compatible, satisfies Kato's Inequality, so in particular we have 
\begin{equation*}
\left\vert D^{2}V\right\vert \geq \left\vert \nabla \left\vert DV\right\vert
\right\vert
\end{equation*}%
whenever $\left\vert DV\right\vert \neq 0$. Hence, in some neighborhood
around $x_{0}$, 
\begin{equation}
\left\vert \nabla \left\vert DV\right\vert \right\vert \leq \frac{C}{r^{2}}
\label{nabladv}
\end{equation}%
for some constant $C>0$. Overall, the time-derivative bound (\ref{ddtDV})
and space derivative bound (\ref{nabladv}) show that there exists some $%
\gamma \in \left( 0,1\right) $ such that for all $\left( x,t\right) \in
P_{\gamma r}\left( x_{0},t_{0}\right) $, 
\begin{equation}
\left\vert DV\left( x,t\right) \right\vert \geq \frac{1}{2r}.
\end{equation}%
Now, for $\theta =t_{0}-\gamma r^{2}$, we have 
\begin{eqnarray*}
\mathcal{F}\left( x_{0},t_{0},\theta \right) &=&\left( t_{0}-\theta \right)
\int_{M}\left\vert DV\left( x,\theta \right) \right\vert
^{2}p_{x_{0},t_{0}}\left( x,\theta \right) \func{vol}\left( x\right) \\
&\geq &\gamma r^{2}\int_{B_{\gamma r}\left( x_{0}\right) }\left\vert
DV\left( x,\theta \right) \right\vert ^{2}p_{x_{0},t_{0}}\left( x,\theta
\right) \func{vol}\left( x\right) \\
&\geq &\frac{1}{4}\gamma \int_{B_{\gamma r}\left( x_{0}\right)
}p_{x_{0},t_{0}}\left( x,\theta \right) \func{vol}\left( x\right) .
\end{eqnarray*}%
However, from Corollary 2.3 of \cite{HamiltonMatrixHarnack}, on $P_{\gamma
r}\left( x_{0},t_{0}\right) $, we have $p_{x_{0},t_{0}}\left( x,\theta
\right) \geq \frac{c}{r^{7}}$ for some constant $c$ that depends only on $M$%
. Therefore, for some $\delta >0$, we do obtain (\ref{Fdelta}).

Now we can proceed with the proof of Theorem \ref{thmEpsRegular}.

\begin{proof}[Proof of Theorem \protect\ref{thmEpsRegular}]
Suppose first that $V$ is a solution of the flow (\ref{heatflow}) on $%
M\times \lbrack 0,t_{0}]$. From the proof of Theorem 3.2 in \cite%
{GraysonHamilton} and from Theorem 3.1 in \cite{HamiltonMatrixHarnack}, we
know that for any $\eta >0$, any constant $C>1,$ and any $x_{0}\in M$ and
any $\tilde{t}_{0}=t_{0}-\alpha \in (0,1]$, there exists a $\rho >0$ such
that for all $\left( \xi ,\tau \right) \in P_{\rho }\left(
x_{0},t_{0}\right) $ 
\begin{equation}
\left( \tau -\alpha \right) p_{\xi ,\tau }\left( x,t\right) \leq C\left(
t_{0}-\alpha \right) p_{x_{0},t_{0}}\left( x,t\right) +\frac{\eta }{2%
\mathcal{E}_{0}}.  \label{heatkernelcomp}
\end{equation}%
Multiplying (\ref{heatkernelcomp}) by $\left\vert DV\right\vert ^{2}$, and
integrating we find 
\begin{equation}
\mathcal{F}\left( \xi ,\tau ,\alpha \right) \leq C\mathcal{F}\left(
x_{0},t_{0},\alpha \right) +\frac{\eta \mathcal{E}\left( \alpha \right) }{2%
\mathcal{E}_{0}}\leq \eta   \label{Feta}
\end{equation}%
as long as $\varepsilon $ in the hypothesis is chosen such that $\frac{\eta 
}{2}\geq C\varepsilon .$ Then, similarly as in the proof of Theorem 3.2 in 
\cite{GraysonHamilton}, define 
\begin{equation*}
q\left( x,t\right) =\min \left\{ \rho -d\left( x_{0},x\right) ,\sqrt{%
t-\left( t_{0}-\rho ^{2}\right) }\right\} .
\end{equation*}%
In some sense this gives the shorter of the distances from $\left(
x,t\right) $ to the spatial boundary of $P_{\rho }\left( x_{0},t_{0}\right) $
and the lower temporal boundary. Now, the function $q\left( x,t\right)
\left\vert DV\left( x,t\right) \right\vert $ will attain its maximum in $%
P_{\rho }\left( x_{0},t_{0}\right) $ at some point $\left( \xi ,\tau \right) 
$ in the interior of $P_{\rho }\left( x_{0},t_{0}\right) $, so that $\sigma
=q\left( \xi ,\tau \right) >0$. Since $\sigma \leq \rho -d\left( x_{0},\xi
\right) \ $and $\sigma ^{2}\leq \tau -\left( t_{0}-r^{2}\right) $, it is
easy to see that $P_{\sigma }\left( \xi ,\tau \right) \subset P_{\rho
}\left( x_{0},t_{0}\right) .$ Moreover, we can also see that $q\left(
x,t\right) \geq \frac{\sigma }{2}\ $on $P_{\frac{\sigma }{2}}\left( \xi
,\tau \right) $. Now, define $r$ such that $\frac{1}{r}=\left\vert DV\left(
\xi ,\tau \right) \right\vert $, then for all $\left( x,t\right) \in P_{\rho
}\left( x_{0},t_{0}\right) ,$ we have 
\begin{equation}
\left\vert DV\left( x,t\right) \right\vert \leq \frac{\sigma }{rq\left(
x,t\right) }.  \label{DVq}
\end{equation}%
Suppose $r\geq \frac{\sigma }{2}.$ Then, since $q\left( x,t\right) \geq 
\frac{\rho }{2}$ on $P_{\frac{\rho }{2}}\left( x_{0},t_{0}\right) $, we
obtain a bound $\left\vert DV\left( x,t\right) \right\vert \leq \frac{4}{%
\rho }$ for all $\left( x,t\right) \in P_{\frac{\rho }{2}}\left(
x_{0},t_{0}\right) .$ Otherwise, suppose $r\leq \frac{\sigma }{2}$. In that
case, from (\ref{DVq}) we find that for all $\left( x,t\right) \in
P_{r}\left( \xi ,\tau \right) $, 
\begin{equation}
\left\vert DV\left( x,t\right) \right\vert \leq \frac{2}{r}.
\end{equation}%
We can now apply Lemma \ref{lemFdelta} to obtain a $\delta >0$ and a $\gamma
\in \left( 0,1\right) $ such that for $\theta =\tau -\gamma r^{2}$ we have%
\begin{equation}
\mathcal{F}\left( \xi ,\tau ,\theta \right) \geq \delta .
\label{Fxitaudelta}
\end{equation}%
Now, by Theorem \ref{thmMonotonicity}, we find that if $\alpha \leq \theta
<\tau ,$ 
\begin{eqnarray}
\mathcal{F}\left( \xi ,\tau ,\theta \right)  &\leq &C\mathcal{F}\left( \xi
,\tau ,\alpha \right) +C\left( \theta -\alpha \right) \left( \mathcal{E}_{0}+%
\mathcal{E}_{0}^{\frac{1}{2}}\right)   \notag \\
&\leq &C\eta +C\left( \theta -\alpha \right) \left( \mathcal{E}_{0}+\mathcal{%
E}_{0}^{\frac{1}{2}}\right) 
\end{eqnarray}%
where we have used (\ref{Feta}). Since $\theta -\alpha \leq t_{0}-\alpha $,
let us find a $\beta $ such that $\beta \geq t_{0}-\alpha ,$ and $C\beta
\left( \mathcal{E}_{0}+\mathcal{E}_{0}^{\frac{1}{2}}\right) <\frac{\delta }{2%
}.$ Choosing $\eta <\frac{\delta }{2C}$ gives us $\mathcal{F}\left( \xi
,\tau ,\theta \right) <\delta $, which contradicts (\ref{Fxitaudelta}).

Thus we find that there exist $\varepsilon >0$ and $\beta $ (where $\beta $
depends on $\mathcal{E}_{0}$), such that for any $\left( x_{0},\alpha
\right) \in M\times \lbrack t_{0}-\beta ,t_{0})$ there exists a $\rho >0$
and a finite $B,$ such that if $\mathcal{F}\left( x_{0},t_{0},\alpha \right)
\leq \varepsilon $, then $\left\vert DV\right\vert \leq B$ is bounded on $%
P_{\rho }\left( x_{0},t_{0}\right) .$ It should be noted that $\rho $ and $B$
only depend on $t_{0}-\alpha $, rather than $t_{0}$ and $\alpha $
individually. Now suppose the solution $V$ only exists on $M\times \lbrack
0,t_{0}).$ Then, by applying the gradient bounds to the translates $\tilde{V}%
\left( x,t\right) =V\left( x,t-\zeta \right) $ and taking $\zeta
\longrightarrow 0$, we obtain uniform bounds on $\left\vert DV\right\vert $
for $t<T$. From Theorem \ref{thmLocDkVest} we then get estimates on higher
derivatives, and thus conclude that the solution extends smoothly to $%
t=t_{0} $ in some neighborhood of $x_{0}$.
\end{proof}

\section{Heat flow in the presence of a torsion-free $G_{2}$-structure}

\setcounter{equation}{0}\label{secrealim} The flow (\ref{heatflow}) and the
octonion covariant derivatives are defined with respect to some fixed
background $G_{2}$-structure that corresponds to the unit octonion $V=1$ (or 
$V=-1$). In fact, due to the covariance of $D$ with respect to change of the
background $G_{2}$-structure (\ref{DtildeAV2}), this choice is arbitrary.
However, if we would like to understand if the flow reaches some particular $%
G_{2}$-structure $\varphi $ within the given metric class, then we can
without loss of generality set the background $G_{2}$-structure to be $%
\varphi ,$ and then all that remains to be checked is whether the flow
reaches $V^{2}=1$ within the maximum time interval $[0,t_{\max }).$
Equivalently, this corresponds to $v=0$ where $v=\func{Im}V$. In this
section we will analyze the behavior of the real and imaginary parts of $V$
along the flow, particularly in the case when a torsion-free $G_{2}$%
-structure exists in the given metric class.

Let $V=f+v$ be the decomposition of the unit octonion $V$ into real and
imaginary parts. Then, we also have $f^{2}+\left\vert v\right\vert ^{2}=1.$
Also, suppose that the initial octonion is given by $V_{0}=f_{0}+v_{0}$. The
background $G_{2}$-structure $\varphi =\sigma _{1}\left( \varphi \right) $
will have torsion $T$ (which we will set to $0$ shortly), the initial $G_{2}$%
-structure $\varphi _{0}=\sigma _{V_{0}}\left( \varphi \right) $ will have
torsion $T_{0}=-\left( DV_{0}\right) V_{0}^{-1}$, and the $G_{2}$-structure $%
\varphi _{V}=\sigma _{V}\left( \varphi \right) $ that corresponds to $V,$
will have torsion $T^{V}=-\left( DV\right) V^{-1}$. Here $D$ is with respect
to $\varphi $.

\begin{lemma}
The evolution of $f$ and $\left\vert v\right\vert ^{2}$ along the flow (\ref%
{heatflow}) is given by 
\begin{subequations}
\begin{eqnarray}
\frac{\partial f}{\partial t} &=&\Delta f+f\left\vert \nabla V\right\vert
^{2}+\left\langle v,\func{div}T\right\rangle  \label{dfdt1} \\
&&+2\left\langle \nabla _{a}V,\left( 1-fV\right) T^{a}\right\rangle  \notag
\\
\frac{\partial \left\vert v\right\vert ^{2}}{\partial t} &=&\Delta
\left\vert v\right\vert ^{2}-2f^{2}\left\vert \nabla V\right\vert
^{2}+2\left\vert \nabla f\right\vert ^{2}-2\left\langle v,f\func{div}%
T\right\rangle  \label{dvsqdt} \\
&&+4f\left\langle \nabla _{a}V,\left( fV-1\right) T^{a}\right\rangle . 
\notag
\end{eqnarray}%
\end{subequations}%
\end{lemma}

\begin{proof}
Taking the inner product of (\ref{heatflow}) with $1$, we get 
\begin{equation}
\frac{\partial f}{\partial t}=\left\langle \Delta _{D}V,1\right\rangle
+\left\vert DV\right\vert ^{2}f.  \label{dfdt}
\end{equation}%
However, 
\begin{eqnarray}
\Delta f &=&\Delta \left\langle V,1\right\rangle  \notag \\
&=&\left\langle \Delta _{D}V,1\right\rangle +2\left\langle
D_{a}V,D^{a}1\right\rangle +\left\langle V,\Delta _{D}1\right\rangle  \notag
\\
&=&\left\langle \Delta _{D}V,1\right\rangle -2\left\langle \nabla
_{a}V-VT_{a},T^{a}\right\rangle -\left\langle V,\func{div}T+\left\vert
T\right\vert ^{2}\right\rangle  \notag \\
&=&\left\langle \Delta _{D}V,1\right\rangle -2\left\langle \nabla
_{a}V,T^{a}\right\rangle -\left\langle v,\func{div}T\right\rangle
+f\left\vert T\right\vert ^{2}
\end{eqnarray}%
and 
\begin{eqnarray}
\left\vert DV\right\vert ^{2} &=&\left\vert \nabla V-VT\right\vert ^{2} 
\notag \\
&=&\left\vert \nabla V\right\vert ^{2}-2\left\langle \nabla
_{a}V,VT^{a}\right\rangle +\left\vert T\right\vert ^{2}.
\end{eqnarray}%
Thus, overall, (\ref{dfdt}) becomes 
\begin{equation*}
\frac{\partial f}{\partial t}=\Delta f+f\left\vert \nabla V\right\vert
^{2}+\left\langle v,\func{div}T\right\rangle +2\left\langle \nabla
_{a}V,\left( 1-fV\right) T^{a}\right\rangle .
\end{equation*}%
Multiplying by $2f,$ we further obtain 
\begin{eqnarray}
\frac{\partial f^{2}}{\partial t} &=&\Delta f^{2}+2f^{2}\left\vert \nabla
V\right\vert ^{2}-2\left\vert \nabla f\right\vert ^{2}+2\left\langle v,f%
\func{div}T\right\rangle  \label{dfsqdt} \\
&&+4f\left\langle \nabla _{a}V,\left( 1-fV\right) T^{a}\right\rangle . 
\notag
\end{eqnarray}%
Since $\left\vert v\right\vert ^{2}=1-f^{2}$, $\frac{\partial \left\vert
v\right\vert ^{2}}{\partial t}=-\frac{\partial f^{2}}{\partial t}$, and
hence we then get (\ref{dvsqdt}).
\end{proof}

\begin{lemma}
Let $u=f^{2}-\left\vert v\right\vert ^{2}=2f^{2}-1$, and suppose $T=0$.
Then, along the flow (\ref{heatflow}), $u$ satisfies the inequality 
\begin{equation}
\frac{\partial u}{\partial t}\geq \Delta u+\left( \frac{u}{1-u^{2}}\right)
\left\vert \nabla u\right\vert ^{2}  \label{dudtineqlem}
\end{equation}
\end{lemma}

\begin{proof}
From (\ref{dfsqdt}), setting $T=0$, we have 
\begin{equation}
\frac{\partial u}{\partial t}=\Delta u+2\left( u+1\right) \left\vert \nabla
v\right\vert ^{2}+2\left( u-1\right) \left\vert \nabla f\right\vert ^{2}.
\label{dudt1}
\end{equation}%
Assume first that $u^{2}\neq 1,$ so that $f\neq 0$ and $v\neq 0$. Since $%
\nabla u=4f\nabla f,$ we find 
\begin{equation}
\left\vert \nabla f\right\vert ^{2}=\frac{1}{16f^{2}}\left\vert \nabla
u\right\vert ^{2}=\frac{1}{8\left( u+1\right) }\left\vert \nabla
u\right\vert ^{2}.
\end{equation}%
With this, (\ref{dudt1}) becomes 
\begin{equation}
\frac{\partial u}{\partial t}=\Delta u+2\left( u+1\right) \left\vert \nabla
v\right\vert ^{2}+\frac{1}{4}\frac{u-1}{u+1}\left\vert \nabla u\right\vert
^{2}.  \label{dudt2}
\end{equation}%
From Kato's inequality, 
\begin{equation*}
\left\vert \nabla v\right\vert ^{2}\geq \left\vert \left( \nabla \left\vert
v\right\vert \right) \right\vert ^{2}=\frac{1}{4\left\vert v\right\vert ^{2}}%
\left\vert \left( \nabla \left( f^{2}\right) \right) \right\vert ^{2}=\frac{1%
}{8\left( 1-u\right) }\left\vert \nabla u\right\vert ^{2},
\end{equation*}%
using which, (\ref{dudt2}) becomes 
\begin{equation}
\frac{\partial u}{\partial t}\geq \Delta u+\left( \frac{u}{1-u^{2}}\right)
\left\vert \nabla u\right\vert ^{2}.  \label{dudt4}
\end{equation}%
It should be noted that generally, Kato's inequality holds whenever $v\neq 0$%
, however in our case, when $v=0,$ $\frac{\partial u}{\partial t}=0$ since $%
u=1-2\left\vert v\right\vert ^{2}.$ However, (\ref{dudt1}) becomes%
\begin{equation}
\frac{\partial u}{\partial t}=\Delta u+4\left\vert \nabla v\right\vert
^{2}=-4\left\vert \left( \nabla \left\vert v\right\vert \right) \right\vert
^{2}+4\left\vert \nabla v\right\vert ^{2}=0.
\end{equation}%
Hence $\left\vert \nabla v\right\vert ^{2}=\left\vert \left( \nabla
\left\vert v\right\vert \right) \right\vert ^{2}$ and thus the inequality
still holds. Now suppose $u=-1,$ so that $f=0$ and hence $\left\vert
v\right\vert =1$. Then, (\ref{dudt1}) becomes 
\begin{equation}
\frac{\partial u}{\partial t}=\Delta u-4\left\vert \nabla f\right\vert
^{2}=0.  \notag
\end{equation}%
On the other hand, in (\ref{dudt2}), as $u\longrightarrow -1$, $\frac{u}{%
1-u^{2}}\left\vert \nabla u\right\vert ^{2}\longrightarrow -4\left\vert
\nabla f\right\vert ^{2},$ so 
\begin{equation*}
0=\frac{\partial u}{\partial t}\geq -\frac{\varepsilon _{2}}{2}\left\vert
\nabla u\right\vert ^{2}
\end{equation*}%
which is of course true. We conclude that (\ref{dudt4}) holds everywhere.
\end{proof}

To be able to apply the Maximum Principle to (\ref{dudtineqlem}) we need to
rewrite it in a different form. This will then allow us to obtain lower
bounds on $u,$ and hence $f$.

\begin{lemma}
\label{lemT0inff}Suppose $T=0$, then along the flow (\ref{heatflow}), $%
f\left( t\right) ^{2}$ is bounded by 
\begin{equation}
\inf_{M}\left[ f\left( t,x\right) \right] ^{2}\geq \inf_{M}\left[ f\left(
0,x\right) \right] ^{2}
\end{equation}%
as long as the flow exists.
\end{lemma}

\begin{proof}
In (\ref{dudtineqlem}), let $u=\sin \theta $ for some function $\theta $
such that $\theta \in \left[ -\frac{\pi }{2},\frac{\pi }{2}\right] $, Then,%
\begin{eqnarray*}
\nabla u &=&\left( \cos \theta \right) \nabla \theta \\
\Delta u &=&-\left( \sin \theta \right) \left\vert \nabla \theta \right\vert
^{2}+\left( \cos \theta \right) \Delta \theta
\end{eqnarray*}%
and hence we can rewrite (\ref{dudtineqlem}) as 
\begin{equation*}
\left( \cos \theta \right) \frac{\partial \theta }{\partial t}\geq \left(
\cos \theta \right) \Delta \theta
\end{equation*}%
Overall, for $-1<u<1,$ and hence $\cos \theta >0$, the inequality (\ref%
{dudtineqlem}) becomes 
\begin{equation}
\frac{\partial \theta }{\partial t}\geq \Delta \theta .  \label{dthetadt}
\end{equation}%
Hence, by the Maximum Principle, we conclude that if $\inf_{M}\left[ f\left(
0\right) \right] ^{2}>0$, then, as long as the flow exists, 
\begin{equation}
\inf_{M}\theta \left( t,x\right) \geq \inf_{M}\theta \left( 0,x\right)
\label{inftheta}
\end{equation}%
and thus, in any case, $\inf_{M}\left[ f\left( t,x\right) \right] ^{2}\geq
\inf_{M}f\left[ \left( 0,x\right) \right] ^{2}$.
\end{proof}

Thus, we have shown that in the presence of a torsion-free $G_{2}$%
-structure, $f^{2}$ is bounded below by its initial value and hence, $%
\left\vert v\right\vert ^{2}$ is bounded above by its initial value. This
shows that pointwise, $V\left( t\right) $ never gets further away from the
torsion-free $G_{2}$-structure than at the initial point. We can do even
better though. It turns out that if initially, $f$ is nowhere zero, then the
integral of $\left\vert f\left( t\right) \right\vert $ increases
monotonically along the flow as long as $V$ is not parallel.

For convenience, we define a new functional 
\begin{equation}
\mathcal{G}\left( t\right) :=\int_{M}\left\vert f\left( t\right) \right\vert 
\func{vol}  \label{Vtdef}
\end{equation}%
which is just the $L^{1}$-norm of $f$ at time $t$. Recall that $\mathcal{E}%
\left( t\right) $ in this case is just the $L^{2}$-norm of $\nabla V$.

\begin{lemma}
\label{thmvL2}Suppose $T=0$, and $k:=\inf_{M}\left\vert f\left( 0,x\right)
\right\vert >0$, then along the flow (\ref{heatflow}), 
\begin{equation}
\frac{\partial \mathcal{G}\left( t\right) }{\partial t}\geq k\mathcal{E}%
\left( t\right) .  \label{dvdtineq3a}
\end{equation}
\end{lemma}

\begin{proof}
Recall from (\ref{dvsqdt}) that along the flow (\ref{heatflow}), for $T=0$, 
\begin{equation}
\frac{\partial f}{\partial t}=\Delta f+f\left\vert \nabla V\right\vert ^{2}
\label{dfdtT0}
\end{equation}%
We know from Lemma \ref{lemT0inff} that in this case, $\inf_{M}f^{2}\left(
t,x\right) \geq \inf_{M}f^{2}\left( 0,x\right) >0,$ so we can rewrite (\ref%
{dfdtT0}) as 
\begin{equation}
\frac{\partial \left\vert f\right\vert }{\partial t}=\Delta \left\vert
f\right\vert +\left\vert f\right\vert \left\vert \nabla V\right\vert ^{2},
\end{equation}%
since $f$ is never zero along the flow. Integrating over $M$, we get 
\begin{equation*}
\frac{\partial \mathcal{G}\left( t\right) }{\partial t}=\int_{M}\left\vert
f\right\vert \left\vert \nabla V\right\vert ^{2}\func{vol}\geq
\inf_{M}\left\vert f\left( t\right) \right\vert \mathcal{E}\left( t\right) 
\end{equation*}%
and hence we get (\ref{dvdtineq3a}).
\end{proof}

\begin{remark}
Lemma \ref{thmvL2}\ shows that as long as initially $f\left( t\right) $ is
nowhere zero (and equivalently $\inf_{M}\left\vert f\left( 0,x\right)
\right\vert >0$), its $L^{1}$ norm is increasing monotonically as long as $%
\mathcal{E}\left( t\right) \neq 0$. Of course, $\left\vert f\right\vert \leq
1$, and so $\mathcal{G}\left( t\right) \leq Vol\left( M\right) .$ Recall
from Lemma \ref{lemfuncevol} that $\mathcal{E}\left( t\right) $ is
decreasing monotonically, with stationary points corresponding to
divergence-free $G_{2}$-structures. In particular, if the flow reaches a
stationary point with $\func{div}T^{\left( V\right) }=0$, but $\mathcal{E}%
\left( t\right) >0$, $\mathcal{G}\left( t\right) $ will still increase. On
the other hand, suppose $M$ has a parallel vector field. An octonion section
which has this vector as the imaginary part and has a constant real part
will then also define a torsion-free $G_{2}$-structure. So if the flow
reaches this section, at that point $\mathcal{E}\left( t\right) $ will
vanish, so it is possible that $\left\vert f\right\vert =1$ will never be
reached in that case, even though a torsion-free $G_{2}$-structure has been
reached. If on the other hand, the flow exists for all $t\geq 0$, then we
see that it will have to converge to a torsion-free $G_{2}$-structure, again
not necessarily the one defined by $\left\vert f\right\vert =1$.
\end{remark}

Combining Corollary \ref{corrGlobalExist} and Lemma \ref{thmvL2}, we obtain
the following theorem. If we assume that a torsion-free $G_{2}$-structure is
given by $U\in \Gamma \left( S\mathbb{O}M\right) $, rather than by the
section $1$, then the condition $\inf_{M}\left\vert f\left( 0,x\right)
\right\vert >0$ is replaced by the condition that initially $\left\vert
\left\langle V\left( 0,x\right) ,U\right\rangle \right\vert >0.$

\begin{theorem}
\label{thmTorsionFree}Suppose $\left( \varphi ,g\right) $ is a $G_{2}$%
-structure on a compact $7$-dimensional manifold $M$. Suppose there exists a
unit octonion section $U$ such that $\sigma _{U}\left( \varphi \right) $ is
torsion free. Then there exists $\varepsilon >0$, such that if the flow (\ref%
{heatflow}) has initial energy density ${\Lambda }_{0}<\varepsilon $ and the
initial octonion section $V_{0}$ satisfies $\left\vert \left\langle
V_{0},U\right\rangle \right\vert >0$ on $M$, then a solution exists for all $%
t\geq 0$, and as $t\longrightarrow \infty $, $V\left( t\right)
\longrightarrow V_{\infty }$ where $V_{\infty }$ defines a torsion-free $%
G_{2}$-structure. If $\left( M,g\right) $ admits no parallel vector fields,
then $V_{\infty }=U$.
\end{theorem}

\begin{proof}
From Corollary \ref{corrGlobalExist} we already know that a solution $%
V\left( t\right) $ will exist for all $t\geq 0$. Recall that we may switch
over to the torsion-free $G_{2}$-structure $\sigma _{U}\left( \varphi
\right) $ as our background $G_{2}$-structure. In particular, from (\ref%
{sigmaprod}), $\sigma _{V\left( t\right) }\left( \varphi \right) =\sigma
_{V\left( t\right) U^{-1}}\left( \sigma _{U}\left( \varphi \right) \right) .$
Hence we can now consider $\tilde{V}\left( t\right) =V\left( t\right) U^{-1}$%
. Let us write $\tilde{V}\left( t\right) =f\left( t\right) +v\left( t\right) 
$. Now the condition $\left\vert \left\langle V_{0},U\right\rangle
\right\vert >0$ is equivalent to $\left\vert f\left( 0\right) \right\vert >0$%
. Thus, from Lemma \ref{thmvL2} we know that $\mathcal{G}\left( t\right) $
is growing monotonically along the flow, however $\mathcal{G}\left( t\right) 
$ is also bounded above by $Vol\left( M\right) $, hence it must converge to
some $\mathcal{G}_{\infty }\leq Vol\left( M\right) $, and in particular this
shows that $\mathcal{E}\left( t\right) \longrightarrow 0$. Hence, the limit $%
V_{\infty }=\lim_{t\longrightarrow \infty }V\left( t\right) $ must define a
torsion-free $G_{2}$-structure. 

If the background $G_{2}$-structure is torsion-free, then the torsion of the 
$G_{2}$-structure defined by a unit octonion $V$ will be given by $T^{\left(
V\right) }=-\left( \nabla V\right) V^{-1}$. Hence torsion-free $G_{2}$%
-structures in the same metric class are given by unit octonion sections $V$
for which $\nabla V=0.$ In particular, the imaginary part of $V$ is then
parallel vector field. So any torsion-free $G_{2}$-structures apart from the
background $G_{2}$-structure are defined by parallel vector fields. Hence,
if there are no parallel vector fields on $M$, then the torsion-free $G_{2}$%
-structure that is compatible with $g$ is unique, and thus $V_{\infty }=U$.
\end{proof}

\section{Concluding remarks}

\label{secConc}\setcounter{equation}{0}The results in this paper are just
the beginning of the study of the heat flow of isometric $G_{2}$-structures
as well its stationary points: $G_{2}$-structures with divergence-free
torsion. In the study of the harmonic map heat flow and the Yang-Mills flow,
results such as monotonicity formulas and $\varepsilon $-regularity led to a
rich study of singularities and solitons of these flows. Clearly, this
should also be possible in our setting, with the interesting added challenge
of interpreting this in terms of the geometry of $G_{2}$-structures. Other
related concepts such as entropy, that have been defined in the harmonic map
and Yang-Mills cases \cite{BolingKelleherStreetsHM1,KelleherStreetsYM1} also
have an analog and interpretation in our case. Some progress in this
direction has been already made in the recent paper \cite{DGKisoflow}.
Another possible direction is to consider the flow in some particular
simpler settings, such as warped product manifolds with $SU\left( 3\right) $%
-structure that have been considered as models for the Laplacian coflow \cite%
{GrigorianSU3flow,KarigiannisMcKayTsui}, in which case the octonion section
should reduce to a unit complex number, or even with $SU\left( 2\right) $%
-structure, in which case the octonion section may reduce to a quaternion
section. Understanding the behavior of the flow in such special settings may
inform further directions of study.

One property of the flow (\ref{heatflow}) that hasn't been fully used yet is
the gauge-invariance, i.e. invariance of the flow under the change of the
background $G_{2}$-structure, as discussed in Section \ref{secHeat}. We used
this in Section \ref{secrealim} to more conveniently describe the behavior
of the flow in the presence of a torsion-free $G_{2}$-structure. In \cite%
{DGKisoflow}, similar ideas were used to show an \textquotedblleft
Uhlenbeck-type trick\textquotedblright , using which the evolution of the
torsion had a more tractable form. It is however likely that this
gauge-invariance can lead to a better understading of the flow.

\bibliographystyle{spmpsc}

\end{document}